\documentclass[final,onefignum,onetabnum]{siamart220329}

\usepackage{epsfig,amsmath}
\usepackage{verbatim,pdfsync}
\usepackage{mathtools,amsfonts} 
\usepackage{mathrsfs}   
\usepackage[utf8]{inputenc}
\usepackage{hyperref}

\usepackage{tikz,eucal,enumerate,mathrsfs,bbm}

\numberwithin{equation}{section}
\numberwithin{figure}{section}

\def\R{\mathbb{R}}
\def\vphi{{\varphi}}

\def\F{{\cal F}}
\def\G{{\cal G}}
\def\D{{\cal D}}
\def\L{{\cal L}}
\def\H{{\cal H}}

\def\A{{\cal A}}

\def\Bbar{{\overline{B}}}

\def\Utilde{{\tilde{U}}}

\def\la{\langle}
\def\ra{\rangle}
\def\pa{{\partial}}
\def\ep{\epsilon}
\def\nn{\nonumber}
\def\vspc{{\vspace{-0.2in}}}

\def\gmax{{g_{max}}}

\def\xidot{{\dot{\xi}}}

\def\xdot{{\dot{x}}}

\def\gammadot{{\dot{\gamma}}}

\def\tWF{{\text{WF}}}
\def\mdetg{{m \sqrt{\det g}}}

\def\gammabar{{\bar{\gamma}}}
\def\mubar{{\bar{\mu}}}

\title{Rigidity in fixed angle inverse scattering for Riemannian metrics}

\author{ 
Lauri Oksanen\thanks{Department of Mathematics and Statistics, University of Helsinki, PO Box 68, 00014 Helsinki,
Finland. ~~Email: lauri.oksanen@helsinki.fi}
\and
Rakesh\thanks{Department of Mathematical Sciences, University of Delaware, Newark DE 19808, USA.
~~ Email: rakesh@udel.edu}
\and
Mikko Salo\thanks{Department of Mathematics and Statistics, P.O. Box 35 (MaD), FI-40014 University of
Jyväskylä, Finland. ~~Email: mikko.j.salo@jyu.fi}
}

\begin{document}

\maketitle

\begin{abstract}
The fixed angle inverse scattering problem for a velocity consists in determining a sound speed, or a Riemannian metric up to diffeomorphism, from measurements obtained by probing the medium with a single plane wave. This is a formally determined inverse problem that is open in general. In this article we consider the rigidity question of distinguishing a sound speed or a Riemannian metric from the Euclidean metric. We prove that a general smooth metric that is Euclidean outside a ball can be distinguished from the Euclidean metric. The methods involve distorted plane waves and a combination of geometric, topological and unique continuation arguments.
\end{abstract}



\section{Introduction}

\subsection{The statement of the problem}

An acoustic medium occupying $\mathbb{R}^n$, $n>1$, with non-constant sound speed, is 
probed by a {\bf single} impulsive plane wave (hence exciting all frequencies), and the far-field medium response is 
measured in all directions for all frequencies. A longstanding open problem, called the fixed angle scattering inverse problem 
(for velocity), is the recovery of the sound speed of the medium from this far-field response. In some situations, the acoustic 
properties of the medium are modeled by a Riemannian metric and then the goal is the recovery of this Riemannian metric 
from the far field measurements corresponding to perhaps $n(n+1)/2$ incoming plane waves. 

We consider time domain, near field
measurement versions of these problems. 
It is not known whether these versions of the problems are equivalent to the 
frequency domain versions of  these problems. For the related operator $\pa_t^2 - \Delta_x + q(x)$, such time domain inverse problems are equivalent to fixed angle inverse scattering problems in the frequency domain \cite{rs20b}.
We show that the near field measurements distinguish between a constant velocity (or Euclidean metric) medium and a 
non-constant velocity (non-Euclidean metric) medium. 
This can be considered as a rigidity theorem analogous to a corresponding result for the boundary rigidity problem 
\cite{gr83}. The significant aspect of our result is that the only condition imposed on the velocity/metric is that they 
are smooth and constant/Euclidean outside a compact set. We do not require a non-trapping condition or a `diffeomorphism 
condition' on the velocity/metric. The results are obtained by combining topological, geometrical and PDE based arguments.

In \cite{ors24a}, we have obtained similar results for the more difficult Lorentzian metric case. Here the data consists
of near field measurements of solutions associated with the operator $\Box_h$ where $h(x,t)$ is a space and time dependent 
Lorentzian metric on $\R^n \times \R$.

In $\R^n$ with $n>1$, $e_i$ denotes the unit vector parallel to the positive $x_i$-axis, $B$ denotes the origin centered open 
ball of 
radius $1$, $B_r(p)$ denotes the $p$ centered open ball of radius $r>0$ and $\Bbar_r(p)$ its closure.
For a curve $s \to \gamma(s)$ in $\R^n$, $\dot{\gamma}(s)$ will denote its derivative. 

For the velocity problem, the waves propagate as solutions of the homogeneous PDE associated with 
the operator $\rho(x) \pa_t^2 - \Delta_x$, corresponding to a 
medium with velocity $ 1/\sqrt{\rho}$, or the operator
$ \pa_t^2 - \Delta_g$ associated to a Riemannian metric $g(x)$ on $\R^n$. Recall that if $g = (g_{ij})$ and
$g^{-1} = (g^{ij})$ then
\[
\Delta_g := \frac{1}{\sqrt{\det g}} \sum_{i,j=1}^n \pa_i \left ( \sqrt{ \det g} \, g^{ij} \, \pa_j \right )
\]
is the Laplace-Beltrami operator associated with the Riemannian metric $g$. 
The operator
\begin{equation}
\L := \pa_t^2 - \frac{1}{ m \sqrt{\det g}} \sum_{i,j=1}^n \pa_i \left ( m \sqrt{ \det g} \, g^{ij} \pa_j \right )
\label{eq:Ldef}
\end{equation}
associated with a positive smooth function $m(x)$ and a smooth Riemann metric $g(x)$ captures both cases. If we take
$m=1$ we have $\L = \pa_t^2 - \Delta_g$ and if we take $g = \rho I$, $m = \rho^{(2-n)/2}$ then
$\L = \pa_t^2 - \rho^{-1} \Delta_x$ which is equivalent to $\rho \pa_t^2 - \Delta_x$ for the homogeneous PDE. 
Note that $\L$ is self-adjoint with weight $m \, \sqrt{\det g}$.

We assume that $m(x)=1$ and $g(x)=I$ for $|x| \geq 1-\delta$ for some small positive $\delta$. 
For a fixed unit vector $\omega$ in $\R^n$, let $U(x,t;\omega)$ be the solution of the initial value problem (IVP)
\begin{subequations}
\begin{align}
\L U(x,t;\omega) =0, & \qquad (x,t) \in \R^n \times \R,
\label{eq:Ude}
\\
U(x,t;\omega) = H(t- x \cdot \omega), & \qquad \text{for } t <-1.
\label{eq:Uic}
\end{align}
\end{subequations}
Here $H(s)$ is the Heaviside function and the initial value for $U$ 
represents an impulsive incoming plane wave moving in the direction $\omega$.
{\em This is a well posed problem whose solution has a trace on $\pa B \times \R$ (see Proposition \ref{prop:forward} and the remarks after it)}. 
For a fixed unit vector $\omega$ in $\R^n$, for a finite set $\Omega$ of unit vectors in $\R^n$, and $T \in \R$,
define the forward maps
\begin{align}
& \F_{\omega,T} : (m,g) \to U(\cdot,\cdot; \omega)|_{\pa B \times (-\infty, T]},
\label{eq:Fomdef}
\\
& \F_{\Omega, T} : (m,g) \to [\F_{\omega,T}(m,g)]_{\omega \in \Omega}
\label{eq:Fdef}
\end{align}
The fixed angle inverse scattering problem is the study of the injectivity, the stability and the inversion of $\F_{\Omega,T}$.

This problem is formally determined because $m(x), g(x)$ depend on $n$ variables and 
$\F_{\Omega,T}(m,g)$ is also a function $n$ parameters. 
Formally determined problems are more difficult than the overdetermined problems
such as the Dirichlet to Neumann map inverse problem for the operator $\L$ where the data depends on 
$2n-1$ variables. Later in this section, we give a detailed survey of the literature for our problem but we 
summarize the earlier results by 
saying that, as far as we know, the only significant past result for our inverse problems is for the $\rho$ problem with $g=I$, by 
Romanov in \cite{rom02}, that $\F$ is injective and stable if $\rho$ is restricted to a small enough 
neighborhood of 
$1$, in the $C^k(\R^n)$ norm, for some $k$ dependent on $n$, and $T$ is large enough.

There are strong results for the 
fixed angle inverse scattering problem for the operator $\pa_t^2 - \Delta + q(x)$, a formally determined problem, but 
our inverse problem is more difficult because, in our problems, the waves move with non-constant speed 
so the solution $U(x,t;\omega)$ has a much more complicated structure than the solution $U(x,t;\omega)$ associated with 
the operator $\pa_t^2 - \Delta + q(x)$. That is why progress so far has been limited to the set of $\rho$ close to $1$.


\subsection{The results}

\begin{definition}
The operator $\L$ defined by \eqref{eq:Ldef} is called {\bf admissible} if $m(x)$ is a smooth positive function on $\R^n$,
$g(x)$ is a smooth Riemannian metric on $\R^n$, and there is a small positive $\delta$ such that
$m(x)=1$ and $g(x)=I$ on $|x| \geq 1-\delta$.
\end{definition}
For admissible $\L$ and for positive functions $\rho$ with $\rho-1$ compactly supported, we define the positive constants 
\begin{align*}
\rho_{min} := \inf_{x \in \R^n} \rho(x), 
&\qquad 
\rho_{max} := \sup_{x \in \R^n} \rho(x),
\\
m_{min} := \inf_{x \in \R^n} m(x), 
&\qquad 
m_{max} := \sup_{x \in \R^n} m(x),
\\
g_{min} := \inf_{x \in \R^n, ~v \in \R^n, ~ |v|=1 } v^T g(x) v,
& \qquad
g_{max} := \sup_{x \in \R^n, ~ v \in \R^n, ~ |v|=1} v^T g(x) v.
\end{align*}
Note that $\rho_{min} \leq 1 \leq \rho_{max}$, $m_{min} \leq 1 \leq m_{max}$ and $g_{min} \leq 1 \leq g_{max}$.

We reserve the subscript $0$ to denote objects associated with the operator $\L$ when it is $\pa_t^2 - \Delta$, that is when
$m=1, g=I$, $\rho=1$. Hence $g_0=I$, $\L_0 := \pa_t^2 - \Delta$ and we use $U_0, \Lambda_0$ (see Proposition
\ref{prop:forward}) to denote the objects associated with $\L_0$.

Our first result is for the operator $ \L = \pa_t^2 - \rho^{-1} \Delta$. 
\begin{theorem}[Distinguishing $\rho$ from $1$]\label{thm:rho}
Consider the admissible operator $\L = \pa_t^2 - \rho^{-1} \Delta$, a real number $T > 4 \sqrt{\rho_{max}} - 1$ and 
$\omega$ a fixed unit vector in $\R^n$. 
If  $\F_{\omega,T}(\rho) = \F_{\omega,T}(1)$ then $\rho =1$.
\end{theorem}
Our second result is for the operator $\L:= \pa_t^2 - \Delta_{g}$.  
%
\begin{theorem}[Distinguishing $g$ from the Euclidean metric]\label{thm:g}
Consider the admissible operator $\L = \pa_t^2 - \Delta_g$, a real number $T > 4 \sqrt{g_{max}} \, - 1$, and let
\begin{equation}
\Omega := \{ e_i : i= 1, \cdots, n \} \cup \{ (e_i + e_j)/\sqrt{2} : i \neq j, ~ i,j=1, \cdots, n \}.
\label{eq:Omegadef}
\end{equation}
If $\F_{\Omega,T}(g) = \F_{\Omega,T}(g_0)$,
there is diffeomorphism $\psi : \R^n \to \R^n$ with $\psi= \text{Id}$ outside $B$ and $g = \psi^*g_0$.
\end{theorem}

We observe that the metric $g$ is made of possibly $n(n+1)/2$ different functions, and correspondingly we employ measurements resulting from incoming waves traveling in $n(n+1)/2$ different directions.
If $\psi: \R^n \to \R^n$ is a diffeomorphism with $\psi=id$ outside $B$ then one can verify that 
$\F_{\omega,T}( \psi^* g) = \F_{\omega,T}(g)$. Hence, one can hope to recover $g$ only up to a diffeomorphism. 
There is no such non-uniqueness for the $\rho$ problem.

Our results are a small step towards perhaps proving the injectivity of $\F_{\Omega,T}$ (up to a diffeomorphism), 
but a significant 
step since we make no assumptions about $\rho$ being close to $1$, or $g$ being close to the Euclidean metric, and we 
do not require a non-trapping condition or the absence of conjugate points or caustics.
%

\subsection{History}

In the literature, for the fixed angle inverse scattering problem, 
the medium is probed by a single incoming wave and the far field data is measured in all directions and at 
all frequencies. We chose to work with time domain data since it permits the use of powerful ideas and tools available for formally determined inverse problems for hyperbolic PDEs in the time domain.
In \cite{rs20b}, we showed that, for the operator $\Box + q(x)$ with $q(x)$ smooth and compactly supported, 
the frequency domain version of the problem and the time domain version are equivalent problems. We have not studied the equivalence of the time and frequency domain data for the {\bf admissible} operators 
$\rho \pa_t^2 - \Delta$ or $\pa_t^2 - \Delta_g$.

The $n=1$ version of the problem for the operator $\rho(x) \pa_t^2 - \Delta_x$ received considerable attention from the 
1950s to 1980s. Here $\rho(x)$ is a smooth positive function on $\R$ with $\rho -1$ supported in
$[0,X]$, $U(x,t)$  is the solution of the IVP
\begin{align*}
(\rho(x) \pa_t^2 - \pa_x^2) U(x,t) =0, & \qquad (x,t) \in \R \times \R
\\
U(x,t) = H(t-x), & \qquad (x,t) \in \R \times (-\infty, 0),
\end{align*}
and the forward map is the reflection data map
\[
\F : \rho \to U(0,\cdot)|_{(-\infty, T]},
\]
for some large $T$. Even though the problem is in one space dimension, it is still a difficult problem to tackle directly 
since the medium has non-constant speed (and is the unknown). Using the travel time change of variables
\[
s(x) = \int_0^x \sqrt{\rho(\xi)} \, d \xi, \qquad x \geq 0,
\]
the original one dimensional inverse problem is transformed to an inverse problem for the constant speed
operator $\pa_t^2 - \pa_s^2 - \sigma(s) \pa_s$, where
\[
\sigma(s) = (1/\sqrt{ \rho})'(x(s)).
\]
The new forward map is the reflection data map
\[
\G : \sigma \to V(0,\cdot)|_{(-\infty, T]}
\]
for some large $T$, where $V(s,t)$ is the solution of the IVP
\begin{align*}
( \pa_t^2 - \pa_s^2 - \sigma(s) \pa_s ) V(s,t) =0, & \qquad (s,t) \in \R \times \R,
\\
V(s,t) = H(t-s), & \qquad (s,t) \in \R \times (-\infty, 0).
\end{align*}
The injectivity, stability, the range and the inversion of $\G$ was resolved in the period 1950s to 1980s through the 
work of several researchers (see \cite{bro00} for the results and a survey), however subtle issues arise when studying
the behavior of the map  $\rho(\cdot) \to \sigma(\cdot)$, needed to translate the results for the operator $ \pa_t^2 - \pa_s^2 - \sigma(s) \pa_s$ to results for the operator $\rho(x) \pa_t^2 - \pa_x^2$.

For the $n>1$ case, as far as we know, the only results addressing the injectivity or stability of $\F$ are for
$\rho$ close to $1$. Under that assumption, the $\vphi$ defined in \eqref{eq:phidef} is a diffeomorphism and 
$U$ has a much simpler structure; further
there are other simplifications because $\rho$ is close to $1$.
The first stability result, which is for $\rho$ close to $1$, is due to Romanov in \cite{rom02}. Even though the result in \cite{rom02} deals only with $\rho$ close to $1$, the problem is still non-trivial as one must deal with solutions whose smooth parts are defined on different regions and one needs estimates on the difference of the smooth parts. \cite{rom02} devised an important idea to tackle this issue. Later, using ideas different from 
the one used in \cite{rom02}, a similar result but with weaker norms was proved in \cite{mps22}. 

Romanov analyzed the problem for the operator $\rho \pa_t^2 - \Delta$ also when $\rho(y,z)$ (with $y \in \R^{n-1}$ and
$z \in \R$) is analytic in $y$ and the corresponding $\vphi$ is a diffeomorphism. In section 3.3 of \cite{rom02}, under the analyticity assumption, he shows 
a provable reconstruction method for recovering
$\rho$ over $0 \leq z \leq z_0$ for some (possibly small) $z_0$, from $\F(\rho)$. He combines ideas from the one dimensional case with power series expansions to obtain the result. The long statement of the result 
may be found in Section 3.3 of \cite{rom02}. Techniques for the one dimensional problem may be adapted to obtain results 
also for the problem where $\rho(y,z)$ is discretized in $y$ as done in  \cite{kr22a}, \cite{kr22b}.

There has been much more progress on studying fixed angle inverse scattering problems for operators with constant
speeds. For the operator $\pa_t^2 - \Delta^2 + q(x)$, using an adaption of the Bukhgeim-Klibanov method introduced
in \cite{bk81} (see \cite{by17} for an exposition), it was shown in \cite{rs20a}, \cite{rs20b} that the map
$q \to [F_+(q), F_-(q)]$ is injective with a stable inverse. 
Here $F_+(q), F_-(q)$ are the fixed angle scattering data for the incoming waves
$H(t- x \cdot \omega)$ and $H(t+ x\cdot \omega)$ respectively. In \cite{mps21}, these ideas were adapted to 
prove similar results for the fixed 
angle scattering problem of recovering the vector field $a(x)$ and the function $q(x)$ from the data associated
with the operator $\pa_t^2 - \Delta + a(x) \cdot \nabla + q(x)$. These ideas were adapted even to the non-constant speed case in \cite{ms22} to prove similar results for the problem of recovering $q(x)$ from the data for the operator
$\pa_t^2 - \Delta_g + q$, where $g$ is a known Riemannian metric on $\R^n$ satisfying certain 
symmetry conditions and has a global convex function. This result does not address the recovery of
the Riemannian metric $g$ from the fixed angle scattering data.

There are other formally determined inverse problems for the operator $\rho \pa_t^2 - \Delta$, arising from other source 
receiver combinations than the one associated with $\F$. \cite{su97} studies the backscattering problem for this operator 
and shows that the backscattering data operator is injective if $\rho$ is close to $1$. They reprove this result in 
\cite{su09} as an application of a generalization of the inverse function theorem. There are also results for the 
Bukhgeim-Klibanov type formally determined inverse problems where the source is an internal source exciting the whole medium initially. Results for these types of problems were the first results for
multidimensional formally determined inverse problems for hyperbolic operators. We state one such problem and result
from \cite{iy03} about the recovery of a `velocity' because the problems are difficult and the ideas may be relevant
for further studies of our problem. The article \cite{su13} also has a result for this Bukhgeim-Klibanov type velocity 
inversion problem.

Suppose $p(x)$ is a smooth positive function on $\Bbar$ and
$u(x,t)$ is the solution of the IBVP
\begin{align*}
u_{tt} - \nabla \cdot ( p\, \nabla u) =0, & \qquad \text{on } B \times [0,T],
\\
u(x,0) = f(x), ~ u_t(x,0) =0, & \qquad x \in B,
\\
u(x,t) = g(x,t), & \qquad (x,t) \in \pa B \times [0,T].
\end{align*}
Here the smooth functions $f$ and $g$ (satisfying the boundary matching condition) are fixed and may be regarded 
as the source for the inverse problem. Note that $1/\sqrt{p(x)}$ is the unknown wave speed associated with the operator.
Under the strong assumption on the source $f$ that
\[
(x-x_0) \cdot (\nabla f)(x) >0, \qquad \forall x \in \Bbar,
\]
for some $x_0$ not in $\Bbar$, in \cite{iy03}, it is shown that the map (for large $T$)
\[
\G : p \to \pa_\nu u|_{\pa B \times [0,T]}
\]
is injective if $p$ is restricted to the set of functions which obey a condition which essentially guarantees that, for some
$\beta>0$, the function $|x-x_0|^2 - \beta t^2$ is a pseudo-convex function for the operator $\pa_t^2 - \nabla \cdot p \nabla$. 

For the operators $\rho(x) \pa_t^2 - \Delta$ (and the operator $\pa_t^2 - \Delta_g$), there is also considerable work on 
overdetermined inverse problems such as the injectivity, stability and inversion of the map
\[
\H : \rho \to \Lambda_\rho
\]
where $\Lambda_\rho$ is the Dirichlet to Neumann map for the operator $\rho(x) \pa_t^2 - \Delta$ over
some space time domain. The Boundary Control method of Belishev introduced in \cite{be87} (see \cite{be07} or \cite{kkl01}
for an exposition), and 
subsequent methods motivated by it, are very effective in showing the injectivity of $\H$ and providing reconstruction 
algorithms, even for problems with non-constant velocity. There are also some weak stability results (see \cite{bkl22})
associated with this method.  We found intriguing the unexpected results (Theorems 3.1, 3.3) in 
\cite{bz14} which suggest that the range of $\H$ is discrete in, what one would consider to be, the appropriate topology for the set 
of $\Lambda_\rho$. The article \cite{suv16} explains the surprising phenomena in \cite{bz14}. 
We also remark that knowledge of the map $\Lambda_{\rho}$ implies knowledge of the scattering relation of the sound speed $\rho$, and this connection has been exploited, for example ~in \cite{su05}, \cite{suv16}, to study wave equation inverse problems.
We also suggest \cite{ko23} which studies an inverse problem for the operator $\rho \pa_t^2 - \Delta$, for various source-receiver combinations so that a certain product of solutions is dense in some function spaces. Some of the results in \cite{ko23} seem to be 
for formally determined problems. We also mention \cite{abn20}, where the recovery of a Riemannian metric from 
the minimal areas enclosed by certain curves is investigated. 

\subsection{Organization}
This article is organized as follows. Section 1 is the introduction. Section 2 states and proves the important geometrical results needed for the proofs of our main results -Theorems \ref{thm:rho}, \ref{thm:g}. 
Section 2 also contains the short proof of the well-posedness of the forward problem. Section 3 contains the proofs of our main 
results. In the Appendix, we state and prove a uniqueness theorem for an IBVP 
for distributional solutions of the wave equation in exterior domains. This is needed in the proofs of 
Theorems \ref{thm:rho},  \ref{thm:g}.

\subsection{Acknowledgements}
L.O.\ was supported by the European Research Council of the European Union, grant 101086697 (LoCal). L.O.~and M.S.~were 
partly supported by the Research Council of Finland, grants 353091 and 353096  (Centre of Excellence in Inverse Modelling and Imaging), 359182 and 359208 (FAME Flagship) as well as 347715. Rakesh’s work was partly funded by grants DMS 1908391 and DMS 2307800 from the National Science Foundation of USA. Views and opinions expressed are those of the authors only and do not necessarily reflect those of the European Union or the other funding organizations. Neither the European Union nor the other funding organizations can be held responsible for them.


\section{The Lagrangian manifold and the well-posedness of the IVP}\label{sec:geodesics}

We recall some standard material about null bicharacteristics for $\L$, geodesics for a Riemannian manifold, and address the well-posedness of the IVP \eqref{eq:Ude}, \eqref{eq:Uic}. The definitions of 
the terms used below may be found in \cite{dui11} and \cite{lee97}.

Fix a unit vector $\omega$. 
Here the $\L$ defined by \eqref{eq:Ldef} is {\em assumed to be an admissible operator. }
We define some of the objects associated with the solution $U$ of \eqref{eq:Ude}, \eqref{eq:Uic}. 
These will be useful in the proofs of Theorems \ref{thm:rho}, \ref{thm:g}.  We associate with $\L$ the Riemannian metric $g$.

\subsection{Notation}

For a Riemannian metric $g$ on $\R^n$ and $v,w \in T_x(\R^n)$, define
\[
\la v, w \ra = v^T g(x) w, \qquad \|v\| = \sqrt{v^Tg(x)v}.
\]
We reserve $v \cdot w$ and $|v|$ for the Euclidean inner product and norm on $\R^n$.
For $x \in \R^n$, $v \in T_x(\R^n)$,
$t \to \gamma_{x,v}(t)$ will denote the geodesic which starts at $x$, with velocity $v$, at time $t=0$.

For a unit vector $\omega$ in $\R^n$, define the hyperplanes
\[
\Sigma_{-,\omega} := \{ x \in \R^n : x \cdot \omega = -1 \},
\qquad
\Sigma_{+,\omega} := \{ x \in \R^n : x \cdot \omega = 1 \}.
\]
A generic point in $\Sigma_{-,\omega}$ will be denoted by $a$. When the context is clear, to avoid cumbersome notation, 
we do not write the dependence on $\omega$; so $\Sigma_{+,\omega}$ and $U(x,t;\omega)$ may be written as
$\Sigma_+$ and  $U(x,t)$.

If $K$ is a subset of $\R^{n+1}$ and $\Lambda$ a subset of $T^*(\R^{n+1})$, we define
\[
\Lambda|_K = \Lambda \cap \{ (x,t; \xi, \tau) \in T^*(\R^n \times \R) : (x,t) \in K \}.
\]
A similar meaning will be given to $\Lambda|_K$ if $K$ is a subset of $\R^n$.


\subsection{Bicharacteristics, geodesics, and the first arrival time function}

\subsubsection{Bicharacteristics and geodesics}

The principal symbol of $\L$ is
\[
p(x,t;\xi,\tau) = -\tau^2 + \xi^T (g(x))^{-1} \xi, \qquad (x,t; \xi, \tau) \in T^*(\R^n \times \R)
\]
and the null bicharacteristics generating (see Proposition \ref{prop:forward}) the wave front set of 
the solution $U(x,t;\omega)$ of  
the IVP \eqref{eq:Ude}, \eqref{eq:Uic} are the solutions of the IVP
\begin{alignat*}{4}
\frac{dx_k}{ds} & =  \frac{\pa p}{\pa \xi_k}  = 2 ( g(x)^{-1} \xi )_k ,  \qquad & \frac{dt}{ds} &= \frac{\pa  p}{\pa \tau} = - 2 \tau, 
\qquad k=1, \cdots, n, 
\\
\frac{d \xi_k}{ds} & = - \frac{ \pa p}{\pa x_k}  = - \xi^T \pa_{x_k} (g(x)^{-1}) \xi ,   \qquad & \frac{d \tau}{ds} &= - \frac{\pa p}{\pa t} =0,
\qquad k=1, \cdots, n,
\end{alignat*}
with the initial conditions
\[
x(0) = a, ~~ t(0) = -1, ~~ \xi(0) = -\tau_0 \omega, ~~ \tau(0) = \tau_0,
\]
for some unit vector $\omega \in \R^n$, $a \in \Sigma_{-,\omega}$ and real $\tau_0 \neq 0$. 
One sees that $p(x(s), t(s); \xi(s), \tau(s))$ is constant along solutions of
this system of ODEs and $0$ when $s=0$, so these solutions are null bicharacteristics of $\L$.

Now $ g(x) g^{-1}(x) = I_n$, so
\[
\pa_k ( g^{-1} ) = - g^{-1} \pa_k (g) \, g^{-1}.
\]
Further, $\tau$ is constant along the solution and $\tau_0 \neq 0$, hence $dt/ds$ is never zero. So these solutions may 
be reparametrized with respect to $t$ and the
relevant null bicharacteristics are the solutions $t \to (x(t,a,\tau), \xi(t,a,\tau))$ of the IVP (here $\cdot$ means $d/dt$)
\begin{subequations}
\begin{align}
\xdot = - \frac{1}{\tau} g^{-1} \xi,
& \qquad \xidot_k = -\frac{1}{2\tau} (g^{-1}\xi)^T \pa_{x_k} (g) (g^{-1}\xi),
\qquad k=1, \cdots, n,
\label{eq:HJ}
\\
x(t=-1,a,\tau) = a, & \qquad \xi(t=-1,a,\tau) = - \tau \omega.
\label{eq:HJic}
\end{align}
\end{subequations}
Note that $x(t,a,\tau)$ is independent of $\tau$ and $\xi(t,a,\tau) = \tau \xi(t,a,1)$. 

Since the solutions are null bicharacteristics, we have $\tau^2 = \xi^T g^{-1}(x) \xi$ and $\tau$ is constant, hence
$|\xi(t,a,\tau)|$ is bounded for any fixed $a, \tau$. Hence, from \eqref{eq:HJ}, $\xdot(t,a,\tau)$ is bounded for a 
fixed $a,\tau$. So solutions of the IVP \eqref{eq:HJ}, \eqref{eq:HJic} exist for all $t \in \R$.

From \eqref{eq:HJ} we see that
\begin{equation}
\xi = -\tau g(x) \xdot,
\label{eq:xixdot}
\end{equation}
hence solutions $[x(t,a,\tau), t; \xi(t,a,\tau), \tau]$ of the IVP \eqref{eq:HJ}, \eqref{eq:HJic}, when projected onto 
$\R^n$, are the solutions of the IVP 
\begin{subequations}
\begin{align}
\frac{d}{dt} ( g \xdot)_k= \frac{1}{2} \xdot^T \pa_{x_k} (g) \xdot, \qquad k=1, \cdots, n,
\label{eq:geodde}
\\
x(t=-1, a) = a, ~~ \xdot(t=-1,a) = \omega,
\label{eq:geodic}
\end{align}
\end{subequations}
Actually, there is a one-to-one correspondence
because solutions of \eqref{eq:geodde}, \eqref{eq:geodic} generate solutions of \eqref{eq:HJ}, \eqref{eq:HJic} through the
relation \eqref{eq:xixdot}. 

Equation \eqref{eq:geodde} is an ODE satisfied by the critical points of the energy 
functional $E(\gamma)$ associated with the Riemannian metric $g$, where
\[
E(\gamma) = \int_c^d \| \gammadot(t)\|^2 \, dt,
\]
and $\gamma : [c,d] \to \R^n$ varies over continuous piece-wise smooth curves with fixed end points.
Further, this equation is also the equation satisfied by the geodesics
of $g$ (the zero acceleration curves) - see \cite[Proposition 39, Chapter 10]{oneill83}. 
Hence the projection of the solutions of
\eqref{eq:HJ}, \eqref{eq:HJic} onto $\R^n$ sets up a one-to-one correspondence between the bicharacteristics with the
initial condition \eqref{eq:HJic} and the geodesics with the initial condition \eqref{eq:geodic}. 

These special bicharacteristics/geodesics start on $\Sigma_{-,\omega}$, at $t=-1$,
with velocity $\omega$, and we label these special null bicharacteristics/geodesics as 
{\bf $\omega$-bicharacteristics/$\omega$-geodesics}. For future use we observe that the $\omega$-geodesics, with the $t$ parametrization, are unit speed curves because
\[
\| \xdot \|^2 = \xdot^T g(x) \xdot = \frac{1}{\tau^2} \xi^T g^{-1}(x) \xi = 1.
\]

\subsubsection{The length functional and the distance metric}

We use the definitions and the results in \cite[Chapters 5,6]{lee97}.  An {\bf admissible curve} $\gamma$ on $\R^n$ is a continuous map $\gamma : [c,d] \to \R^n$ with a partition 
$c=c_0 < c_1 < \cdots < c_m=d$ such that, for each
$i=1, \cdots, m$, $\gamma(r)$ is smooth on $[c_{i-1}, c_{i}]$ with one sided derivatives of all orders at $c_{i-1}, c_i$, and
{\bf $\gammadot(r)$ is never zero on $[c_{i-1},c_i]$}. The length, in the Riemannian metric, of an admissible curve $\gamma$ is
 \[
L(\gamma) :=  \int_c^d  \| \gammadot(r) \| \, dr;
\]
note that $L(\gamma)$ is invariant under a smooth increasing reparametrization of $\gamma$.

As in \cite{lee97}, one can define admissible variations of an admissible curve $\gamma$ and seek
critical points of $\L$ for admissible variations but with fixed end points. 
Further, from \cite[Corollary 6.7]{lee97}, every fixed end critical point of $L$ is a smooth curve 
and, when reparametrized to have constant speed, is a geodesic of $g$. Conversely, every geodesic is a critical point of $L$; note that geodesics have constant speed (in the Riemannian sense) and any affine reparametrization still results in
a geodesic.

For points $p,q \in \R^n$, define the Riemannian distance
\[
d(p,q) := \inf \{ L(\gamma) : \text{$\gamma$ an admissible curve in $\R^n$ from $p$ to $q$} \}.
\]
From \cite[Chapter 6]{lee97} we know that $\R^n$ is a metric space with the metric $d(p,q)$ .
Further, using the line segment joining $p$ to $q$, we note that
\begin{equation}
 \sqrt{g_{min}} \, |q-p| \leq d(p,q) \leq  \sqrt{g_{max}} \, |q-p|, \qquad p,q \in \R^n,
 \label{eq:ddist}
\end{equation}
so $(\R^n, d)$ is a complete metric space. 
Hence, by \cite[Corollary 6.15]{lee97} of the Hopf-Rinow theorem, for any $p,q \in \R^n$, there is a 
geodesic (so a smooth curve) $\gamma$ joining $p$ to $q$ such that
\[
L(\gamma) = d(p,q).
\]

 \subsubsection{The time of first arrival function}
 
 Now we restrict attention to $\omega$-geodesics. Note that the $\omega$-geodesics, at time $t=-1$, are at
 the point $a \in \Sigma_{-,\omega}$, have velocity $\omega$ - so are orthogonal to $\Sigma_{-,\omega}$, 
 and have unit speed for all $t$.

Since $\L$ is admissible, $\L = \Box$ on the region $|x| \geq 1- \delta$, so in particular on the region $x \cdot \omega \leq -1$.
With some effort one can show (we do not use it) that the solution $U(x,t)$, of the IVP \eqref{eq:Ude}, \eqref{eq:Uic}, is
zero if $t$ is smaller than the time of first arrival, at $x$, of any of the $\omega$-geodesics. This suggests  the definition of
a candidate for the time of first arrival function.

\begin{definition}\label{def:alphadef} Given a unit vector $\omega$ in $\R^n$, the {\em time of first arrival function} 
$\alpha_\omega: \R^n \to \R$ is defined as 
\[
\alpha_\omega(x) := 
\begin{cases}
x \cdot \omega & \qquad \text{if } x \cdot \omega < -1,
\\
-1 +  \inf \{ d(x, a) : a \in \Sigma_{-} \} & \qquad \text{if } x \cdot \omega \geq -1.
 \end{cases}
\]
\end{definition}
\vspc
\noindent
Note that 
\[
\alpha_\omega(x) = x \cdot \omega, \qquad \text{if } x \cdot \omega \leq -1 + \delta.
\]
To keep the notation simple, we write $\alpha$ instead of $\alpha_\omega$ when we work with a 
fixed $\omega$.

Since the metric $d(\cdot, \cdot)$ on $\R^n$ is topologically equivalent to the Euclidean metric on $\R^n$ 
(because of \eqref{eq:ddist}),
the usual continuity and compactness argument shows that, for points $x$ in the region $ x \cdot \omega \geq -1$, 
the infimum is attained in the definition of $\alpha(x)$. Hence, given an $x$ with $x \cdot \omega \geq -1$, there is an
$a_x \in \Sigma_-$ such that
\[
\alpha_\omega(x) = -1 + d(x, a_x).
\]
Hence, as discussed earlier, there is a geodesic $\gamma$ joining $a_x$ to $x$ such that
\[
\alpha_\omega(x) = -1 + d(x, a_x) = -1 + L(\gamma).
\]
Since the geodesic $\gamma$ is a Riemannian distance minimizing curve from $\Sigma_-$ to $x$, it is orthogonal to
$\Sigma_-$ in the Euclidean/Riemannian sense (they are the same near $\Sigma_-$), hence its velocity there is
parallel to $\omega$ and points in the positive $\omega$ direction. So if we reparametrize $\gamma$ to have unit speed
and start on $a_x$ at time $t=-1$ then $\gamma$ will be an $\omega$-geodesic with
\[
\gamma(t) = \gamma_{a_x, \omega}(t+1).
\] 
Summarizing, for every $x$ in the region
$ x \cdot \omega \geq -1$, there is a segment of an $\omega$-geodesic $\gamma$, from $\Sigma_-$ to $x$ such that
\[
\alpha_\omega(x) = -1 + L(\gamma).
\]

We note that $\alpha_\omega$ is Lipschitz continuous, hence differentiable a.e., because using the distance 
minimizing definition of  $\alpha_\omega(x)$ one can see that for $x, y \in \R^n$, if $\gamma$ is a line segment from 
$x$ to $y$ then
\begin{align*}
\alpha_\omega(x) \leq \alpha_\omega(y) + L(\gamma) \leq \alpha_\omega(y) + \sqrt{g_{max}} \, |y-x|,
\\
\alpha_\omega(y) \leq \alpha_\omega(x) +  L(\gamma) \leq \alpha_\omega(x) + \sqrt{g_{max}} \, |y-x|.
\end{align*}
Hence
\[
|\alpha_\omega(y) - \alpha_\omega(x)| \leq \sqrt{g_{max}} \, |y-x|.
\]

\subsection{The well-posedness of the IVP}

Define the maps $\Phi:  \R \times \Sigma_- \times ( \R \setminus \{0\}) \to T^*(\R^n \times \R )$ with
\begin{equation}
\Phi(t,a, \tau) = (x(t,a), \, t; \xi(t,a, \tau), \, \tau ),
\label{eq:Phidef}
\end{equation}
and, its projection onto $\R^n$, the map $\vphi: \R \times \Sigma_-  \to \R^n$ with
\begin{equation}
\vphi(t,a) = x(t,a) = \gamma_{a,\omega}(t+1).
\label{eq:phidef}
\end{equation}
Let $\Lambda$ be the range of $\Phi$. So $\Lambda$ is a subset of $T^*(\R^n \times \R)$ with
\begin{equation}
\Lambda := \{ (x(t,a,\tau), t; \xi(t,a,\tau), \tau) : t \in \R, ~ a \in \Sigma_{-}, ~ \tau \in \R, ~ \tau \neq 0 \}.
\label{eq:lambdadef}
\end{equation}
We mention in passing - we do not need this claim - that $\Lambda$ is a conic Lagrangian submanifold of 
$T^*(\R^n \times \R)$. See \cite[Chapters 3,4]{dui11} for definitions and the proof.

We state the well-posedness result for the IVP \eqref{eq:Ude}, \eqref{eq:Uic}.
\begin{proposition}[The forward problem]\label{prop:forward}
For the admissible $\L$ defined by \eqref{eq:Ldef}, the IVP \eqref{eq:Ude}, \eqref{eq:Uic} has a unique distributional solution 
$U$ and $\tWF(U)=\Lambda$.
\end{proposition}

\noindent
Here WF$(U)$ denotes the wave front of $U$ - see \cite[Definition 1.3.1]{dui11} for the definition of the wave front set 
of a distribution. Note that $\tWF(U)=\Lambda$ and not just $\tWF(U) \subset \Lambda$; this will be critical in the proofs of the theorems below. 
\begin{proof}
The existence of a distributional solution may be proved by construction, following the method in \cite[Chapter 4]{dui11}. Also,
a careful proof of the existence, in a more general setting, is a part of the proof of \cite[Proposition 5.1]{ors24a}.

Since $U = H(t-x \cdot \omega)$ for $t<-1$, from the definition of $\Lambda$ we have 
$\Lambda|_{t<-1} = \text{WF}(U)|_{t<-1}$. By H{\"o}rmander's theorem on the 
propagation of singularities (see \cite[Theorem 6.1]{dh72}), WF$(U)$ is invariant under the null bicharacteristic flow for $\L$. 
Since every null bicharacteristic of $\L$ intersects the region $\{t<-1\}$, WF$(U)$ is the flow out of $\Lambda|_{t<-1}$ under the 
null bicharacteristic flow of $\L$. Hence WF$(U) = \Lambda$ by the definition of $\Lambda$.

To prove the distributional solution is unique, suppose $V$ is a distributional solution of $\L V=0$ on $\R^n \times \R$
with $V=0$ for $t<-1$. Then repeating the argument in the previous paragraph, we conclude that WF$(V)$ is the empty set, that is $V$ is smooth on $\R^n \times \R$. Then, using standard energy estimates for $V$, on conical regions 
$\sqrt{g_{min}} \, |x-x_0| \leq (t-t_0)$, for arbitrary $(x_0, t_0) \in \R^n \times \R$, leads to $V=0$ on $\R^n \times (-\infty, T)$; it is the usual argument used to prove the domain of dependence result.
\end{proof}

Now WF$(U)=\Lambda$, and $\Lambda$ does not intersect the normal bundle of $\pa B \times \R$. Hence, an application 
of \cite[Proposition 1.3.3]{dui11}
to the inclusion map $\pa B \times \R \to \R^n \times \R $ shows that
$U$ has a trace as a distribution on $\pa B \times \R$. Hence the $\F_{\Omega,T}$ defined by \eqref{eq:Fomdef} is 
well defined.

$\Lambda$ is the subset of $T^*(\R^n \times \R)$ traced out by the $\omega$-bicharacteristics.
The properties of $\Lambda, \Phi$ and $\vphi$ give us detailed structural information about the solution of the IVP 
\eqref{eq:Ude}, \eqref{eq:Uic}. We do not state their additional important properties in this article as they are not needed here. However, these will be important for future attempts to study the injectivity of $\F_{\Omega,T}$.


\subsection{ The diffeomorphism condition}

The next proposition is crucial in the proofs of Theorems \ref{thm:rho}, \ref{thm:g}. It gives a useful condition guaranteeing that
$\vphi$ (defined in \eqref{eq:phidef}) restricted to $ (-\infty, T) \times \Sigma_-$ is a diffeomorphism. 
The proposition is a little stronger than what we need for the proofs of the theorems, 
and parts (a) - (c) are of independent interest.

\begin{proposition}\label{prop:diffeo}
Suppose $g$ is a smooth Riemannian metric with $g=I$ on $|x| \geq 1- \delta$ for some small positive
$\delta$, $\omega$ is a unit vector in $\R^n$, and $T>  2\sqrt{g_{max}} -1$. If $\vphi$ restricted to the set 
$\vphi^{-1}(\R^n \setminus B) \cap \{ t <T\}$ is injective, the following holds.
\begin{enumerate}[(a)]
\item Every $\omega$-geodesic crosses $\Sigma_+$ before time $T$ and never crosses it again.
\item The set $W_T := \vphi( (-\infty, T) \times \Sigma_-  )$ is an open subset of $\R^n$ and contains the region 
$x \cdot \omega < -1 + (T+1)/\sqrt{\gmax}$.
\item $\vphi$ restricted to $(-\infty, T) \times \Sigma_-$ is a diffeomorphism onto $W_T$.
\item $\alpha_\omega(x)$ is a smooth function on $W_T$ and $\| \nabla_g \alpha_\omega\|=1$ on $W_T$, where 
$\nabla_g \alpha_\omega = g^{-1} \nabla \alpha_\omega$. Further $( \nabla_g \alpha_\omega)(x)$ is the velocity of the 
$\omega$-geodesic through $x$.
\item We have $\pi (\Lambda|_{W_T}) = \{ (x,t=\alpha_\omega(x)) : x \in W_T \}$ where $\pi$ is the projection 
$(x,t; \xi, \tau) \to (x,t)$.
\end{enumerate} 
Here $\alpha_\omega$ is the time of first arrival function in Definition \ref{def:alphadef}.
\end{proposition}
Since $T> 2 \sqrt{\gmax} -1$, from (b) we have that $W_T$ contains the region $x \cdot \omega \leq 1$, in particular $W_T$ contains $\Bbar$.
\begin{figure}[h]
\begin{center}
\epsfig{file=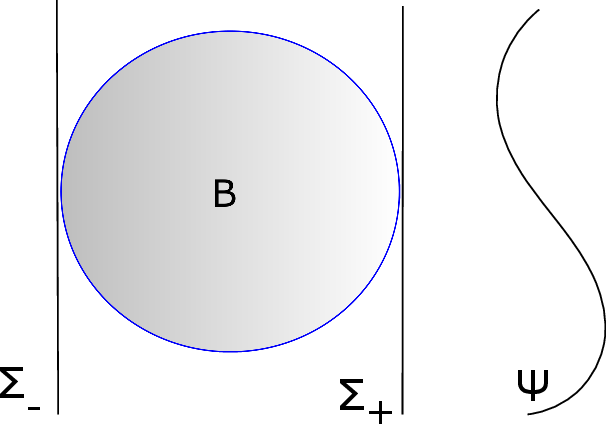, height=1in}
\end{center}
\caption{$\Psi$ is the boundary of $W_T$ and $W_T$ consist of everything on the left of $\Psi$.}
\end{figure}

\begin{proof}
Recall that
\[
\vphi(t,a) = \gamma_{a,\omega}(t+1), \qquad t \in \R, ~ a \in \Sigma_-.
\]
Since $\omega$ is fixed, to keep notation manageable, we use the symbol $\alpha$ instead of $\alpha_\omega$.
\begin{enumerate}[(a)]

\item First we show that every $\omega$-geodesic reaches $\Sigma_+$ before time $T$. 

Let $x_*$ be a point in the region $x \cdot \omega < -1 + (T+1)/\sqrt{\gmax}$. From the properties of 
$d(\cdot, \cdot)$ mentioned above, there is an $a_* \in \Sigma_-$ such that
\[
d(a_*,x_*) = \min_{a \in \Sigma_-} d(a, x_*)
\]
and the segment of the $\omega$-geodesic $ t \to \gamma_{a_*,\omega}(t+1)$ from $a_*$ to $x_*$ has length 
$d(a_*,x_*)$. Further, the geodesic has unit speed and, by \eqref{eq:ddist}, this minimum distance is at most 
$\sqrt{\gmax} (x_* \cdot \omega + 1)$ so the geodesic reaches $x_*$ at least by time 
\[
t=  \sqrt{\gmax} (x_* \cdot \omega + 1) - 1 < T 
\]
Hence the region $x \cdot \omega < -1 + (T+1)/\sqrt{\gmax}$ lies in $W_T$. In particular, since 
$T> 2 \sqrt{\gmax} -1$, $W_T$ contains the region $x \cdot \omega \leq 1$.

Now $\Sigma_+$ lies in $W_T \setminus B$ and $\vphi$ is injective on 
\[
\vphi^{-1} (\R^n \setminus B) \cap \{t<T \} = \vphi^{-1}( W_T \setminus B) \cap \{t<T\}.
\]
Thus, for every point $p \in \Sigma_+$, there is exactly one $\omega$-geodesic
which reaches $p$ before time $T$ - actually it reaches by time $2 \sqrt{\gmax}-1$. Further, noting that if $a \in \Sigma_-$ and 
$|a+\omega| > 1$,  the $\omega$-geodesics $t \to \gamma_{a,\omega}(t+1)$ are lines parallel to $\omega$, the injectivity 
hypothesis
implies any 
$\omega$-geodesic reaching $\Sigma_+ \cap {\Bbar}_1(\omega)$ before time $T$ must have originated at a point 
$a \in \Sigma_- \cap {\Bbar}_1(-\omega)$. 
So if we define
\[
K = \vphi^{-1}(\Sigma_+ \cap \Bbar_1(\omega) ) \cap \{-1 \leq t \leq 2 \sqrt{g_{max}} - 1 \},
\]
then $K$ is closed and contained in $ [ -1, 2 \sqrt{g_{max}} - 1 ] \times \Bbar_1(-\omega)$ hence compact, and the
map $F : K \to \Sigma_+ \cap \Bbar_1(\omega)$ with
\[
F(t,a) = \vphi(t,a) = \gamma_{a,\omega}(t+1)
\]
is a continuous bijection, hence a homeomorphism. 

So there are continuous maps $A : \Sigma_+ \cap \Bbar_1(\omega) \to \Sigma_- \cap \Bbar_1(-\omega) $ and 
$\tau : \Sigma_+ \cap \Bbar_1(\omega) \to \R$ with
\[
\vphi( \tau(x), A(x) ) = x, \qquad x \in \Sigma_+ \cap \Bbar_1(\omega),
\]
that is
\[
\gamma_{A(x), \omega}(\tau(x)+1) = x, \qquad x \in \Sigma_+ \cap \Bbar_1(\omega).
\]
We show that $A$ is surjective, which will prove that all $\omega$-geodesics reach $\Sigma_+$ before time $T$.

First we observe that if $\gamma$ is an $\omega$-geodesic which reaches $\Sigma_+$ then it will never cross 
$\Sigma_+$ again. This is so because if $\gamma$ reaches $\Sigma_+$ at time $t_*$ then 
$\gammadot(t_*) \cdot \omega >0$ because if $\gammadot(t_*) \cdot \omega \leq 0$ then running $\gamma$ 
backwards from time $t_*$ we will never reach $\Sigma_-$ since all geodesics are straight lines outside $B$. 

Next we show that $A$ is injective. Suppose $A(x') = A(x'')$ for some $x', x'' \in \Sigma_+ \cap \Bbar_1(\omega)$ with 
$x' \neq x''$. Let $a= A(x')=A(x'')$ and $t' = \tau(x'), ~ t'' = \tau(x'')$. Then $a \in \Sigma_-$, $0 < t', t'' <T$ and
\[
x' = \vphi(a,t'), \qquad x'' = \vphi(a,t'').
\]
Since $x' \neq x''$ we must have $t' \neq t''$. So there is an $\omega$-geodesic which crosses $\Sigma_+$ at two different times, which contradicts the assertion in the previous paragraph. Hence $A$ is injective.

If $\A$ is the range of $A$, then $A : \Sigma_+ \cap \Bbar_1(\omega) \to \A$ is a continuous bijection, hence a homeomorphism
since $\Sigma_+ \cap \Bbar_1(\omega)$ is compact. Hence $\A$ has the same (trivial) homology as 
$\Sigma_+ \cap \Bbar_1(\omega)$. It is clear that
\[
A(x) = x - 2 \omega, \qquad \text{if } x \in \Sigma_+ \cap \left ( \Bbar_1(\omega) \setminus B_{1-\delta}(\omega) \right ),
\]
so $\A$ contains an annulus, that is
\[
\Sigma_- \cap \left [ B_1(-\omega) \setminus B_{1-\delta}(-\omega) \right ] \, \subset \A \, \subset \Sigma_-.
\]
Hence, if any point of $\Sigma_- \cap B_1(- \omega)$ were missing from $\A$, the $(n-2)$-th homology group of $\A$ would be non-zero\footnote{
Define $S= \Sigma_- \cap \{ x : |x-\omega|=1 \}$; note $S \subset \A$. Since there is a $p \notin \A$ with
$p \in \Sigma_- \cap \{ x : |x-\omega|<1 \}$, there is a (retraction) continuous map $r : \A \to S$ with $r(x)=x$ for $x \in S$. 
So $ r \circ i = id$ where $i : S \to \A$ is the inclusion map. Hence, $r_* \circ i_* = id$ where $i_*, r_*$ are the induced maps
$i_* : H_{n-2}(S) \to H_{n-2}(\A)$ and $r_*: H_{n-2}(\A) \to H_{n-2}(S)$.  Since $H_{n-2}(S) \neq (0)$, we must have 
$H_{n-2}(\A) \neq 0$. The retraction $r$ is the map which sends $x \in \A$ to the point $r(x)$ on $S$ where the ray from $p$ through $x$ intersects $S$.
}.
So $\A = \overline{B}_1(-\omega) \cap \Sigma_-$. 

Since we already know the behavior of the $\omega$-geodesics starting outside $\Sigma_- \cap B_1(-\omega)$,
we conclude that all $\omega$-geodesics cross $\Sigma_+$ before time $T$ and never cross it again.

\item We prove (b) and (c) together. 

We have already shown above that $W_T$ contains the region $ x \cdot \omega < -1+(T+1)/\sqrt{\gmax}$. Also,
from (a) we note that if $\gamma$ is an $\omega$-geodesic and it crosses $\Sigma_+$ at $x$, then $\gamma$ minimizes
the distance of $x$ from $\Sigma_-$ and there are no other $\omega$-geodesics through $x$.

We first show that $\vphi$ restricted to $(-\infty, T) \times \Sigma_-$ is injective. Suppose $p \in B$ and there are two
$\omega$-geodesics $\gamma_1, \gamma_2$ through $p$. Suppose $\gamma_i$ originates at $a_i$ in $\Sigma_-$ and
crosses $\Sigma_+$ at $x_i$, $i=1,2$. The lengths of the segments of $\gamma_1, \gamma_2$ from 
$a_1, a_2$ to $p$ must be same otherwise we would be able to create a curve from one of the $a_i$ to the other $x_j$ 
whose length is shorter than the shortest distance from $\Sigma_-$ to $x_j$. With this equality, 
consider the curve $\gamma$ consisting of the part of $\gamma_1$ from $a_1$ to $p$ followed by the part of 
$\gamma_2$ from $p$ to $x_2$. Hence we have a curve from $\Sigma_-$ to $x_2$ which has the same length as 
the shortest distance from $\Sigma_-$ to $x_2$. Hence $\gamma$ is a distance minimizing 
curve from $\Sigma_-$ to $x_2$ so, by \cite[Corollary 26, Chapter 10]{oneill83}, $\gamma$ must be a geodesic 
which is normal to $\Sigma_-$. In particular $\gamma$ must be smooth. Hence $\gamma_1$ and $\gamma_2$ must have the same velocity at $p$, hence $\gamma_1=\gamma_2$.
A similar but simpler argument works if there is an $\omega$-geodesic which goes through $p$ at two different times. Finally, because of the injectivity hypothesis, it is clear that, for $p$ outside $B$, there is at most one $\omega$-geodesic reaching
$p$ before time $T$. Hence $\vphi$ restricted to $(-\infty, T) \times \Sigma_-$ is injective.

Next, we show that $\vphi$ restricted to $ (-\infty, T) \times \Sigma_-$ is a local diffeomorphism, that is $\varphi'$ is 
non-singular. For convenient notation, 
we take $\omega=e_n$ - the unit vector parallel to the $x_n$ axis. So the vector fields $\pa_{a_i}$, $i=1, \cdots, n-1$, form a basis 
for the tangent space to $\Sigma_-$, at every point on $\Sigma_-$. 
Suppose $\varphi'$ is singular for some $(t^*,a^*)$ with $t^*<T$; then we have vectors
$\delta a \in \R^{n-1}$ and $\delta t \in \R$ such that $[\delta a, \delta t] \neq 0$ and
\begin{equation}
\sum_i \delta a_i \, \pa_{a_i} \varphi + \delta t \, \pa_t \varphi =0, \qquad \text{at } (t^*,a^*).
\label{eq:tempasts}
\end{equation}
Now $(\pa_t \varphi)(t,a)  = {{\gammadot}}_{a,\omega}(t+1)$ so $\|\pa_t  \varphi \|=1$; further we claim that 
\begin{equation}
g( \pa_{a_i}\varphi, \pa_t \varphi ) =0, \qquad i=1, \cdots, n-1.
\label{eq:git}
\end{equation}
Assuming this for the moment, taking the $g$ inner product of \eqref{eq:tempasts} with $\pa_t \varphi$, we see that 
$\delta t=0$, hence $\delta a \neq 0$, $t^*<T$, and
\begin{equation}
\sum_i \delta a_i \, \pa_{a_i} \varphi =0, \qquad \text{at } (t^*, a^*).
\label{eq:aiphii}
\end{equation}
We postpone the proof of \eqref{eq:git} to the end of this item and continue with the proof of the claim that $\varphi'$ is invertible, which will require appealing to some results for Lorentzian manifolds.

We regard $\R^n \times \R$ as a Lorentzian manifold with the metric $ - dt^2 + g(x)$ and we identify $\Sigma_-$ with 
the spacelike submanifold $\Sigma_- \times \{t=-1\}$ - we continue to call it $\Sigma_-$ for convenience. 
One can show that $\gamma$ is a unit speed geodesic for $(\R^n,g)$ iff $ t \to \gammabar(t)= [\gamma(t), t]$ is a 
null geodesic for the Lorentzian manifold. Further, $[\omega, 1]$ is normal to $\Sigma_-$ in this metric, so $t \to [\gamma_{a,\omega}(t+1),t]$
are null geodesics normal to $\Sigma_-$. The map $\psi : (-\infty, T) \times \Sigma_ - \to \R^n \times \R$ with
\[
\psi(t,a) = (\gamma_{a,\omega}(t+1),t) = (\varphi(t,a), t)
\]
is part of the normal exponential map for $\Sigma_-$, and $\psi'$ is singular at $(t^*, a^*)$ because of \eqref{eq:aiphii}. 
Let $\gamma$ be the unit speed normal geodesic $t \to \gamma(t) = \gamma_{a^*,\omega}(t+1)$ and 
$\gammabar$ the corresponding null normal gedoesic
$t \to \gammabar(t) = (\gamma(t),t)$. Then, from \cite[Proposition 30, Chapter 10]{oneill83},  the point 
$p = \gammabar(t^*)$ is a focal point of $\Sigma_-$ along $\gammabar$.  Pick any point $q=(x^{**}, t^{**})$ on $\gammabar$
with $t^*< t^{**}<T$. Then, from \cite[Proposition 48, Chapter 10]{oneill83}, applied to $P = \Sigma_-$, there is a 
timelike curve $\mubar$ from $\Sigma_- $ to $q$; $\mubar$ will be future pointing. We parametrize $\mubar$ with time so $\mubar(t) = (\mu(t),t)$ with $\|\dot{\mu}(\cdot)\|<1$. 

We restrict $\gamma(\cdot)$ and $\mu(\cdot)$ to the interval $[-1, t^{**}]$. Then 
$\gamma, \mu$ go from $\Sigma_-$ to $x^{**}$ and
$L(\mu)< L(\gamma)$ because
\[
L(\mu) = \int_{-1}^{t^{**}} \|\dot{\mu}(t)\| \, dt < t^{**}+1,
\qquad
L(\gamma) = \int_{-1}^{t^{**}} \| \dot{\gamma}(t)\| \, dt = t^{**}+1.
\]
So the shortest distance (w.r.t $g$) from $\Sigma_-$ to $x^{**}$ is less than $L(\gamma)$. Hence there is geodesic, different from
$\gamma$, from $\Sigma_-$ to $x^{**}$, which is normal to $\Sigma_-$, implying there are at least two $\omega$-geodesics
from $\Sigma_-$ to $x^{**}$,  reaching $x^{**}$ before time $T$. This contradicts what we proved above - that $\varphi$ is injective.

So $W_T = \vphi( (-\infty, T) \times \Sigma_-)$ is an open subset of $\R^n$ and $\vphi: (-\infty, T)\times \Sigma_- \to W_T$
is a local diffeomorphism and a bijection. Hence $\vphi$ is a diffeomorphism.

It remains to prove the claim \eqref{eq:git}. Now $\varphi(t,a) = \gamma_{a,\omega}(t+1)$, so $t \to \varphi(t,a)$ is a geodesic which is normal to $\Sigma_-$. Recall that we have taken $\omega=e_n$; let $D$ denote covariant derivative (see 
\cite[Lemma 4.9]{lee97}). Then, from \cite[Lemma 6.3]{lee97}, we have $D_t \pa_{a_i} \varphi = D_{a_i} \pa_t \varphi$. Hence from \cite[Lemma 5.2c]{lee97} we have
\begin{align*}
\pa_t \la \pa_{a_i} \varphi, \pa_t \varphi \ra & = \la D_t \pa_{a_i} \varphi, \pa_t \varphi \ra
+  \la \pa_{a_i} \varphi, D_t \pa_t \varphi \ra
= \la D_{a_i} \pa_t \varphi, \pa_t \varphi \ra
\\
& = \frac{1}{2} \pa_{a_i} \| \pa_t \varphi \|^2 = 0,
\end{align*}
where we used the facts that $D_t \pa_t \varphi =0$ and $ \| \pa_t \varphi \|=1$ because $t \to \varphi(t,a)$ is a 
unit speed geodesic. Now, at $t=-1$, we have $\la \pa_i \varphi, \pa_t \varphi \ra = \la e_i, \omega \ra =0$, so \eqref{eq:git} holds.

\item[(d)]
The inverse of the restricted $\vphi$ in (c) is the map
\[
W_T \ni x \to (\tau(x), A(x)) \in (-\infty, T) \times \Sigma_-,
\]
for some smooth functions $A(x), \tau(x)$. Since $\vphi(a,t) = \gamma_{a,\omega}(t+1)$ and the $\omega$-geodesics are 
unit speed curves, one sees that $\alpha( \vphi(a,t)) = t$, so $\alpha(x)=\tau(x)$ hence smooth.

Again
\[
\alpha( \gamma_{a,\omega}(t) ) = t,
\]
so if we construct, as in the proof of Theorem 2 in Section 3.2 of \cite{evans10}, the solution $\psi(x)$ of the first 
order non-linear PDE $ (\nabla \psi)^T g(x) \, (\nabla \psi) =1$ with 
$\psi(x) = x \cdot \omega$ on $x \cdot \omega = -1$, one sees that $\psi(x) = \alpha(x)$ and
$(\nabla_g \alpha)(x) := g(x)^{-1} (\nabla \alpha)(x)$ is the speed of the $\omega$-geodesic through $x$.
Hence (d) has been proved.

\item[(e)] Noting that $\alpha(\vphi(t,a)) = t$, we have
\[
\pi(\Lambda|_{W_T}) = \{ (\vphi(t,a), t ) : (t,a) \in (-\infty, T) \times \Sigma_- \}
= \{ ( x, \alpha(x)) : x \in W_T \}.
\]
\end{enumerate}
%
\end{proof}


The following proposition has the same conclusion as Proposition \ref{prop:diffeo} but a simpler proof of (a) because of the 
stronger hypothesis. This stronger hypothesis holds in the proofs of Theorems 
\ref{thm:rho} and \ref{thm:g}.

\begin{proposition}\label{prop:diffeosalo}
Suppose $g$ is a smooth Riemannian metric with $g=g_0$ on $|x| \geq 1- \delta$ for some small positive
$\delta$, and $\omega$ is a unit vector in $\R^n$. If $T> 2 \sqrt{g_{max}} -1$ and 
\begin{equation}
\Lambda_g|_{ (\R^n \setminus B) \times (-\infty, T)} 
= \Lambda_0|_{ (\R^n \setminus B) \times (-\infty, T)} 
\label{eq:lamlam0}
\end{equation}
then (a)-(e) of Proposition \ref{prop:diffeo} hold.
\end{proposition}
{\em Remark.} It will be clear from the proof that the hypothesis of Proposition \ref{prop:diffeosalo} implies the hypothesis
of Proposition \ref{prop:diffeo}, hence (a)-(e) of Proposition \ref{prop:diffeo} will hold. 
What we gain from the stronger hypothesis of Proposition \ref{prop:diffeosalo} is a simpler proof of (a). 

\begin{proof}
The proof of (a) is more direct than the proof of (a) in Proposition \ref{prop:diffeo}. The proofs of (b) - (e) are the same as in
Proposition \ref{prop:diffeo}, so we will not repeat them. 

When $g=I$, then from \eqref{eq:HJ}, \eqref{eq:HJic}, $x(t,a), \xi(t,a,\tau)$ are the solutions of the IVP
\[
\dot{x} = - \xi/\tau, ~~ \dot{\xi} = 0; \qquad x(t=-1,a) = a, ~~ \xi(t=-1,a,\tau) = - \tau \omega.
\]
Hence $x(t,a) = a + (t+1) \omega$, $\xi(t,a, \tau) = - \tau \omega$. So, from \eqref{eq:lambdadef}, we have
\begin{align*}
\Lambda_0 = \{ (a + (t+1) \omega, t; - \tau \omega, \tau) : t \in \R, ~ a \in \Sigma_-, ~ \tau \in \R, ~ \tau \neq 0 \}.
\end{align*}
For every $x \in \R^n$, there is a unique $a \in \Sigma_-$ and $t \in \R$ such that $x = a + (t+1) \omega$; in fact
$t = x \cdot \omega$ and $a = x - (t+1) \omega$, hence
\[
\Lambda_0 = \{ (x, x \cdot \omega; - \tau \omega, \tau) : x \in \R^n, ~ \tau \in \R, ~ \tau \neq 0 \}.
\]
So, from the hypothesis, we have that
\[
\Lambda_g|_{ (\R^n \setminus B) \times (-\infty, T)} = \{ (x, x \cdot \omega; -\tau \omega, \tau) : x \cdot \omega < T, ~ x \notin B,
\tau \in \R \},
\]
hence
\[
\vphi( (-\infty, T) \times \Sigma_- ) \setminus B = \{ x \in \R^n: x \cdot \omega < T, ~ x \notin B \}.
\]
So, through each point $x$ in this region there is an $\omega$-geodesic which reaches $x$ at time $x \cdot \omega$ and, at 
that instant, has velocity $\omega$ (through the relation $\xi = - \tau g(x) \, \xdot$), and any $\omega$-geodesic through $x$ will have velocity $\omega$ at that point. Since the position and the velocity of a 
geodesic at that position determine the geodesic, through each point $x$ in the above region there is a unique 
$\omega$-geodesic and it reaches $x$ at time $x \cdot \omega$ and has velocity $\omega$ at that time. Hence
$\vphi$ is injective on the region $ \vphi^{-1} ( \R^n \setminus B) \cap \{ t < T \}$ and satisfies the hypothesis of Proposition
\ref{prop:diffeo}.

From the previous paragraph we can conclude that for each $x \in \Sigma_+$, there is a unique $a_x \in \Sigma_-$
such that the $\omega$-geodesic $t \to \gamma_{a_x, \omega}(t+1)$ reaches $x$ before time $T$ - in fact it reaches
$x$ at time $x\cdot \omega=1$ and with velocity $\omega$. Hence
\[
a_x = \gamma_{x,\omega}(-2).
\]
and we have the smooth map $A: \Sigma_+ \to \Sigma_-$ with
\[
A(x) =  \gamma_{x, \omega}(-2) = a_x, \qquad x \in \Sigma_+.
\]
We claim that $A$ is a diffeomorphism.

$A$ is injective because if $A(x')=A(x'')=a$ for some $x', x'' \in \Sigma_+$ and $a \in \Sigma_-$ then, from our work in the previous paragraph,
\[
x' = \gamma_{a,\omega}(2) = x''.
\]

Next we show that $A'$ is invertible at each point of $\Sigma_+$. Let $x_* \in \Sigma_+$ and $(t_*, a_*)$ the unique point in 
$(-\infty, T) \times \Sigma_- $ with $\gamma_{a_*, \omega}(t_*+1) = x_*$ - also $t_*=1$ and 
$\gammadot_{a_*, \omega}(t_*+1) =\omega$.  Hence the equation
\[
\gamma_{a,\omega}(t+1) \cdot \omega = 1
\]
has a solution $(t_*, a_*)$ and, since
\[
\gammadot_{a_*,\omega}(t_*+1) \cdot \omega = \omega \cdot \omega = 1 \neq 0,
\]
by the implicit function theorem, we have a smooth map $a \to \tau(a)$ defined in some neighborhood $N$ (in $\Sigma_-$) 
of $a_*$ with $\tau(a_*)=t_*$ and
\[
 \gamma_{a,\omega}(\tau(a)+1) \cdot \omega = 1, \qquad a \in N.
\]
So we have a smooth map $\psi : N \to \Sigma_+$ with
\[
\psi(a) = \gamma_{a,\omega}(\tau(a)+1) = \gamma_{a,\omega}(2)
\]
such that $\psi \circ A$ is the identity map on $N$, hence $A'(x_*)$ is invertible.

We now prove the surjectivity of $A$. Since $A$ is a local diffeomorphism the range of $A$ is open. Also,
\[
A(x) = x - 2 \omega, \qquad x \in \Sigma_+, ~ |x-\omega| \geq 1.
\]
so the range of $A$ is
\[
A \left ( \Sigma_+ \cap  \bar{B}_1(\omega)  \right )  \cup \{ a \in \Sigma_-: |a+\omega| \geq 1 \}.
\]
The first set is compact, hence closed, and the second set is closed. Hence the range of $A$ is closed, open and
non-empty so must be $\Sigma_-$.
%
\end{proof}
 
 \section{Proofs of the theorems} 
 
A major component of the proofs of Theorem \ref{thm:rho} and Theorem \ref{thm:g} is the proposition below, which is proved
for the general $\L$ defined by \eqref{eq:Ldef}. 
 \begin{proposition}\label{prop:Delta}
 Let $\L$ be the admissible operator defined by \eqref{eq:Ldef}, $\omega$ a fixed unit vector and $U$ the solution of the IVP \eqref{eq:Ude}, \eqref{eq:Uic}.  If $T> 4 \sqrt{\gmax}-1$ and 
 \[
 U|_{\pa B \times (-\infty, T)} =  U_0|_{\pa B \times (-\infty, T)} ,
 \]
 (the time of first arrival function) $\alpha(x)$ is defined and smooth on the region $x \cdot \omega < T$, 
 $\alpha(x) = x \cdot \omega$ outside $B$, and
 \[
 \Delta_{m,g} \alpha := \frac{1}{\mdetg} \sum_{i,j} \pa_i ( \mdetg \, g^{ij} \, \pa_j \alpha ) =0, \qquad \text{on the region $x \cdot \omega <T$}.
 \]
 \end{proposition}
We postpone the proof of Proposition \ref{prop:Delta} to the end of this section and show how the theorems follow from this proposition.
 
 
 \subsection{Proof of Theorem \ref{thm:rho}}
 
 Here $\L = \pa_t^2 - \rho^{-1} \Delta_x$ which corresponds to $g = \rho I$ and $m=\rho^{(2-n)/2}$ in \eqref{eq:Ldef} and
 we use the $\alpha$ associated to this $\L$. We also note that $g_{max} = \rho_{\max}$. Hence (see Proposition \ref{prop:Delta} for the definition) $\Delta_{m,g} = \rho^{-1}\Delta$.
 
Proposition \ref{prop:Delta} implies $\alpha(x)$ is defined and smooth on the region $x \cdot \omega < T$ and 
\[
\Delta \alpha =0~~ \text{on } x \cdot \omega < T,
\]
and the size of $T$ implies that this region contains a neighborhood of $\Bbar$. Now $\alpha(x) = x \cdot \omega$ outside
$B$ therefore, from the uniqueness of solutions of the BVP for Laplace's equation, we have
\[
\alpha(x) = x \cdot \omega, \qquad \text{on } \Bbar
\]
Hence $\nabla_g \alpha = g^{-1} \nabla \alpha = \rho^{-1} \omega$, so $\|\nabla_g \alpha\|^2=1$ implies
\[
1 = \la \nabla_g \alpha, \nabla_g \alpha \ra = \rho^{-1}, \qquad \text{on } \Bbar,
\]
hence $\rho=1$ on $\Bbar$.

 
 \subsection{Proof of Theorem \ref{thm:g}}

 Here we have $\L = \pa_t^2 - \Delta_g$ corresponding to $m=1$, $\Omega$ the set of unit vectors defined 
 by \eqref{eq:Omegadef} and let
 \[
 W := \{ x \in \R^n : x \cdot \omega < T ~~ \text{for each } \omega \in \Omega \};
 \]
 note that $W$ is an open set containing  $\Bbar$. From Proposition \ref{prop:Delta}, for each $\omega \in  \Omega$, 
 we have the function $\alpha_\omega(x)$ (constructed in Section \ref{sec:geodesics}), defined and smooth at least on $W$, with
 \[
 \Delta_g \alpha_\omega =0, \qquad \text{on } W,
 \qquad  \alpha_\omega(x) = x \cdot \omega \qquad x \in W \setminus B.
 \]
 Extend $\alpha_\omega$ smoothly to all of $\R^n$ with $\alpha_\omega(x) = x \cdot \omega$ for $x$ outside $B$ and, since $g=g_0$ outside $B$, we observe that
 \begin{equation}
 \Delta_g \alpha_\omega =0 ~~ \text{on } \R^n,
 \qquad  \alpha_\omega(x) = x \cdot \omega ~~ \text{for } x \in \R^n \setminus B.
 \label{eq:Deltaw}
 \end{equation}
 We also note that (d) in Proposition \ref{prop:diffeo} gives us
 \begin{equation}
 \| \nabla_g \alpha_\omega\| =1 ~~ \text{on } \R^n.
 \label{eq:dalpha}
 \end{equation}
 So \eqref{eq:Deltaw} and \eqref{eq:dalpha} hold for each $\omega \in \Omega$.
 
 For $i \neq j$,
 \[
 \alpha_{(e_i+e_j)/\sqrt{2}}(x) = x \cdot (e_i + e_j)/\sqrt{2} = ( \alpha_{e_i} + \alpha_{e_j})(x)/\sqrt{2},
 \qquad \text{for } x \in \R^n \setminus B,
 \]
 so from \eqref{eq:Deltaw} and the uniqueness of solutions of the BVP for the operator $\Delta_g$, we 
 conclude that
 \[
  \alpha_{(e_i+e_j)/\sqrt{2}} = \frac{ \alpha_{e_i} + \alpha_{e_j}}{\sqrt{2}}, \qquad \text{on } \R^n.
 \]
 Therefore \eqref{eq:dalpha} implies that
 \[
 1 = \| \nabla_g  \alpha_{(e_i+e_j)/\sqrt{2}} \|^2 = \frac{1}{2} \| \nabla_g \alpha_{e_i} + \nabla_g \alpha_{e_j} \|^2
 = 1 +  \la \nabla_g \alpha_{e_i}, \nabla_g \alpha_{e_j} \ra,
 \]
 hence
 \begin{equation}
 \la \nabla_g \alpha_{e_i} , \nabla_g \alpha_{e_j} \ra = \delta_{ij}, \qquad i,j=1, \cdots n.
 \label{eq:ortho}
 \end{equation}
 
 Define the smooth map $\psi : \R^n \to \R^n$ with
 \[
 \psi(x) = ( \alpha_{e_1}(x), \cdots, \alpha_{e_n}(x) ), \qquad x \in \R^n.
 \]
 Then \eqref{eq:ortho} implies $\psi'(x)$ is invertible for each $x \in \R^n$ and $\psi(x) = x$ for $x$ outside $B$. Hence, for 
 any compact subset $K$ of $\R^n$, the continuity of $\psi$ implies $\psi^{-1}(K)$ is closed, and $\psi(x)=x$ for $x$ outside 
 $B$ implies $\psi^{-1}(K)$ is bounded; so $\psi$ is a proper map. Hence Hadamard's theorem 
 (see \cite[Theorem 2.1]{rusu15}) implies $\psi$ is a diffeomorphism. Further \eqref{eq:ortho} may 
 be rewritten as $\psi'(x) g(x)^{-1} \psi'(x)^T = I$ so $g = \psi^* g_0$.


\subsection{Proof of Proposition \ref{prop:Delta}}

We note that $U_0= H(t- x \cdot \omega)$ and, as shown in the proof of Proposition \ref{prop:diffeosalo} we have
 \begin{align*}
 \Lambda_0 & =  \{ (x,x \cdot \omega; - \tau \omega, \tau : x \in \R^n, ~ \tau \in \R \}.
\end{align*}

Define
 \[
 V(x,t) = U(x,t) - U_0(x,t), \qquad (x,t) \in \R^n \times (-\infty, T).
 \]
 Since $\L=\L_0=\Box$ on $\R^n \setminus B$ and $ U=U_0$ on $\pa B \times (-\infty, T)$, we have
 \begin{align*}
 \Box V =0, & \qquad \text{on } (\R^n \setminus B) (-\infty, T),
 \\
 V=0, & \qquad \text{on } (\R^n \setminus B) \times (-\infty, 0)
 \\
 V=0, & \qquad \text{on } \pa B \times (-\infty, T).
 \end{align*}
 Noting that 
 \begin{align*}
 \tWF(V) & \subset \tWF(U) \cup \tWF(U_0) = \Lambda_0 \cup \Lambda 
 \\
 & \subset \{ (x,t;\xi, \tau) : \tau^2 = |\xi|^2
 ~ \text{or} ~ \rho(x) \, \tau^2 = |\xi|^2, ~ \tau \neq 0 \},
 \end{align*}
 the hypotheses of Proposition \ref{prop:exterior} are satisfed. 
 Hence, by Proposition \ref{prop:exterior},  we have $V=0$ on $(\R^n \setminus B) \times (-\infty, T)$, that is
 \begin{equation}
 U(x,t) = U_0(x,t)=H(t-x \cdot \omega), \qquad \text{on } (\R^n \setminus \Bbar) \times (-\infty, T),
 \label{eq:UU0}
 \end{equation}
 which implies 
 \begin{align*}
 \Lambda \cap \{ |x| \geq 1, ~ t  < T \} & = \Lambda_0 \cap \{ |x| \geq 1, ~ t < T \}
 \\
& =  \{ (x, x \cdot \omega; - \tau \omega, \tau) : |x| \geq 1, x \cdot \omega < T, ~ \tau \in \R \}.
 \end{align*}
 Noting the relation $\xi = -\tau \, \gammadot$ (see \eqref{eq:xixdot}) between the velocity of $\omega$-geodesics $\gamma$ 
 and the corresponding point in $\Lambda$ and $\Lambda_0$, through every point $x$ in 
 $(\R^n \setminus B) \cap \{ x \cdot \omega < T \}$,
 there is a unique $\omega$-geodesic which reaches $x$ at time $x \cdot \omega$ and its velocity that instant is $\omega$; 
 the uniqueness holds because the position and the velocity of a geodesic, at a given time, uniquely determines the 
 geodesic. In particular,
 \begin{equation}
 \vphi ( (-\infty, T) \times \Sigma_- ) \cap (\R^n \setminus B) =  \{ x \cdot \omega < T \} \cap (\R^n \setminus B).
 \label{eq:rangephi}
 \end{equation}
 Hence $\vphi$ is injective on $\vphi^{-1} ( \R^n \setminus B) \cap \{ t < T \}$, so the hypothesis of Proposition \ref{prop:diffeo} (and also Proposition \ref{prop:diffeosalo}) is satisfied. Therefore, from Proposition \ref{prop:diffeo}, 
 \[
 W_T := \vphi \left ( (-\infty, T) \times \Sigma_- \right ) 
 \]
 is an open subset of $\R^n$  which contains the region $x \cdot \omega < -1 + (T+1)/\sqrt{\gmax}$ (in particular the region
$ \Bbar$) and $\vphi : (-\infty, T) \times \Sigma_- \to W_T$ is a diffeomorphism. Further $\alpha(x)$ is smooth on $W_T$. Combining the above with \eqref{eq:rangephi} and noting that
 \[
 \alpha(\vphi(t,a)) = t,
 \]
 we see that 
 \[
 W_T = \{ x : x \cdot \omega < T \} \qquad \text{and} \qquad  \alpha(x) = x \cdot \omega ~~ \text{in }  W_T \setminus B.
 \]
 Since $T> 4 \sqrt{\gmax}-1$ and $\gmax \geq 1$, we see that $W_T$ includes the region $x \cdot \omega \leq 3$.
 
 We also observe from (e) in Proposition \ref{prop:diffeo} and that $\Lambda_g=\Lambda_0$ on $(\R^n \setminus B) \times (-\infty, T)$ that $U$ is smooth on the region $\R^n \times (-\infty, T) \setminus \{ (x, \alpha(x)) : x\cdot \omega <T \}$.

 
Our intermediate goal is to show that
\begin{equation}
U(x,t) = H(t-\alpha(x)), \qquad \text{on } \R^n \times (-\infty, T).
\label{eq:UHalpha}
\end{equation}
If we have proved this, then noting that 
\[
\L (H(t-\alpha(x)) ) = \left ( 1 - \nabla \alpha^T g^{-1} \, \nabla \alpha \right ) \, \delta'(t-\alpha(x))
+ (\L \alpha) \delta(t-\alpha(x))
= (\L \alpha) \delta(t-\alpha(x)),
\]
\eqref{eq:UHalpha} and \eqref{eq:Ude} imply that $\L \alpha =0$ on $t= \alpha(x)$. Therefore
\begin{equation}
\frac{1}{\mdetg} \pa_i \left ( \mdetg \, g^{ij} \, \pa_j \alpha \right ) = 0, \qquad \text{on ~ } \R^n.
\label{eq:lapalpha}
\end{equation}
and the proof of Proposition \ref{prop:Delta} would be complete. 

It remains to prove \eqref{eq:UHalpha}. Define the regions
\begin{align*}
Q_+ & := \R^n \times (-\infty, T) \cap \{ (x,t) : x \cdot \omega <T, ~t>\alpha(x)\},
\\
Q_- & := \R^n \times (-\infty, T) \cap \{(x,t) : x \cdot \omega <T, ~ t < \alpha(x) \}.
\end{align*}
One sees that \eqref{eq:UHalpha} follows immediately from the following two claims:
\begin{gather}
U=1 ~~ \text{on } Q_+, \qquad U=0 ~~ \text{on } Q_-;
\label{eq:U1U0}
\\
U \in L^2_{loc}( \R^n \times (-\infty, T))
\label{eq:L2loc}
\end{gather}

\noindent
\underline{Proof of \eqref{eq:U1U0}}

From \eqref{eq:UU0}, we already know this to be true for points which lie outside $B \times (-\infty, T)$;
it remains to prove this for points in $B \times (-\infty, T)$. This will be done via a standard unique continuation 
argument across the timelike surface $\pa B \times (-\infty, T)$, followed by an energy estimate. We carry out this for points in
$Q_+$ and the argument for $Q_-$ is similar.
 
Observe the {\bf crucial fact} that $\L 1=0$ on $\R^n \times \R$ 
hence, if $v=U-1$, then $v$ is {\bf smooth} on $Q_+$ and
\begin{equation}
\L v =0 ~~ \text{on } Q_+, \qquad v =0 ~~ \text{on } Q_+ \cap ( (\R^n \setminus B) \times (-\infty, T) ).
\label{eq:vdebc}
\end{equation}
Let
\[
\tau := 3 \sqrt{\gmax} -1;
\]
one can verify that 
\[
 \max_{x \in \Bbar} \alpha(x) < \tau < T.
\] 
We build a special family of non-characteristic (w.r.t $\L$) surfaces, and unique continuation across these surfaces will prove that $v=v_t=0$ on $B \times \{t=\tau\}$.
\begin{figure}
\begin{center}
\epsfig{file=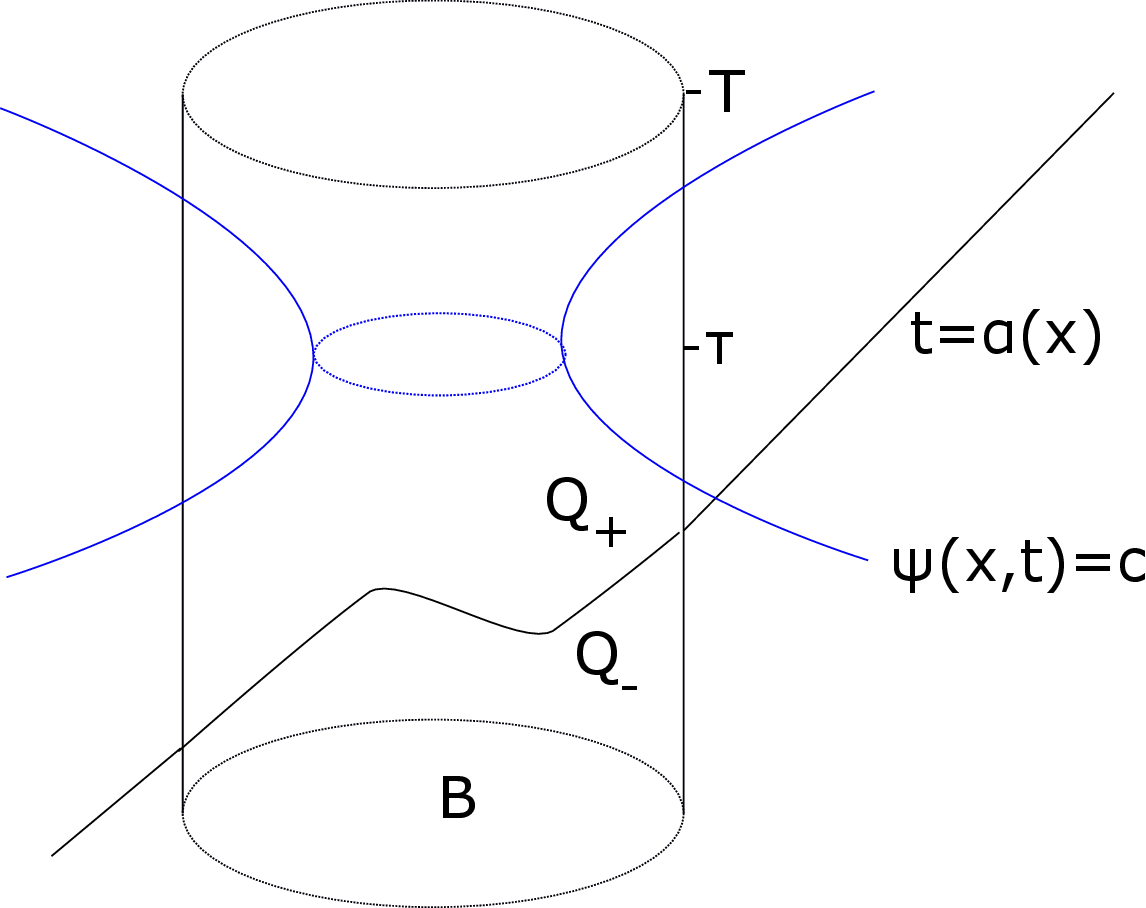, height=1.4in}
\end{center}
\caption{The level surfaces of $\psi$}
\end{figure}
Define
\[
\psi(x,t) := \gmax |x|^2 - (t- \tau)^2;
\]
For any $c>0$, the level surface $\psi(x,t)=c$ is non-characteristic w.r.t $\L$ because the principal symbol of 
$\L$ is $-\tau^2 + \xi^T g^{-1} \xi$ and
\begin{align*}
\frac{ \psi_t^2}{\psi_x^T g^{-1} \psi_x} & = \frac{1}{\gmax^2} \, \frac{ (t-\tau)^2}{ x^T g(x)^{-1} x }
\leq \frac{ (t- \tau)^2}{\gmax |x|^2 }
= \frac{ (t- \tau)^2}{c + (t- \tau)^2}
<  1.
\end{align*}
The region $\psi \geq 0$ is the conical region $|t-\tau| \leq \sqrt{\gmax} \, |x|$ and one can check that
\[
\max_{x \in \Bbar} \alpha(x) < \tau - \sqrt{\gmax}, \qquad \tau + \sqrt{\gmax} < T.
\]
Hence, for $0<c \leq g_{max}$, the level surfaces $\psi=c$ intersect the boundary of $Q_+ \cap B \times (-\infty,T)$
%
%
only on the lateral surface $(\pa B \times \R) \cap Q_+$, and there is an open set $N \subset \R^n \times \R$ containing
$B \setminus \{0\} \times \{t=\tau\}$ such that
\[
N  \subset \{ (x,t): 0<\psi(x,t) < g_{max} \}  \, \cap \, Q_+
\]
Also, one may check that the level surface $\psi = g_{max}$ is outside the region $B \times \R$. 
Noting that the coefficients of $\L$ are independent of $t$, applying the unique continuation Robbiano-Tataru theorem (see Theorem 2.66 in \cite{kkl01}) to $v$ over the region $Q_+$, there is unique continuation across $\psi=c>0$ for $\L$. 
Hence
\begin{equation}
v(\cdot,\tau)=0, ~~ v_t (\cdot,\tau)=0, \qquad \text{on } \Bbar.
\label{eq:vvtzero}
\end{equation}

Let 
\[
\Omega_+ = \{ (x,t) : x \in \Bbar, ~ \alpha(x) < t < T \}.
\]
We already know that $\L v=0$ on $\Omega_+$ and $v=0, \pa_{\nu} v=0$ on the lateral boundary of $\Omega_+$. We also
have the zero `initial data' condition \eqref{eq:vvtzero}. Hence, by Proposition \ref{prop:energy}, we have $v=0$ on $Q_+$,
that is $U=1$ on $Q_+$.


\noindent
\underline{Proof of \eqref{eq:L2loc}}

Let $\chi$ be a smooth function on $\R^n \times (-\infty, T)$ with 
\begin{itemize}
\item $\chi=1$ on a neighborhood of $ \Bbar \times [-1, T]$, 
\item $\chi=0$ for $t \leq -3$.
\end{itemize}
Define 
\[
\Utilde := \chi U;
\]
then, using $\pa_0$ for $\pa_t$, the summation convention, and \eqref{eq:UU0}, on $\R^n \times (-\infty, T)$,
we have
\begin{align*}
\L (\Utilde ) & = [\L, \chi] U =
 \left ( c_{ij} (\pa_i \chi)  \pa_j  + (\L \chi) \right ) U
 =  \left ( c_{ij} (\pa_i \chi)  \pa_j  + (\L \chi) \right ) H(t- x \cdot \omega)
\\
& = a(x,t) \delta(t- x \cdot \omega) + b(x,t) H(t- x \cdot \omega)
\\
& =: F(x,t),
\end{align*}
for some smooth functions $a,b,c$ determined by $m,g,\chi$. Now $F \in H^{-1}_{loc}(\R^n \times (-\infty, T))$ because 
$\delta(t-x \cdot \omega) = \pa_t ( H(t- x \cdot \omega) )$.
Since $\Utilde$ is the solution of the IVP
\begin{align*}
\L \Utilde = F, & \qquad \text{on } \R^n \times (-\infty, T),
\\
\Utilde = 0, & \qquad \text {on } t \leq -3,
\end{align*}
from \cite[Theorems 23.2.4]{horm07} for IVP for hyperbolic PDEs, we can
conclude\footnote{
We take $X = \R^n \times (-\infty, T)$, $\phi=t-3$, and $Y$ an aribtrary ball in $\R^n \times (-\infty, T)$. Then 
Theorem 23.2.4 guarantees a solution in $L^2_{loc}(\R^n \times (-\infty, T))$. This solution can be shown to be unique 
by repeating the uniqueness argument in the proof of Proposition \ref{prop:forward}. Hence $\Utilde$ is this unique solution
in $\L^2_{loc}(\R^n \times (-\infty, T))$. 
}
that $\Utilde \in L^2_{loc}(\R^n \times (-\infty, T))$. 
Next $1-\chi=0$ on a neighborhood of $\Bbar \times [-2, T)$, hence
$(1-\chi) U = (1-\chi) H(t- x \cdot \omega)$. Therefore
\[
U = \chi U + (1-\chi) U = \Utilde + (1-\chi) H(t- x \cdot \omega) \in L^2_{loc}(\R^n \times (-\infty, T)).
\]


\appendix

\section{Uniqueness for an exterior IBVP problem}

We prove a uniqueness result for distributional solutions of an exterior IBVP. An alternative proof based on propagation of singularities may be found in \cite[Lemma 6.1]{ors24a}.

Points $x \in \R^n$ will sometimes
be written as $x=(y,z)$ with $y \in \R^{n-1}$ and $z \in \R$, and we define the hyperplane
\[
H := \{ (y,z) \in \R^n : z=0 \}.
\]

\begin{proposition}\label{prop:exterior}
Suppose $\Omega$ is an open subset of $\R^n$ with $(\R^n \setminus B) \subset \Omega$ and
$V(x,t)$ is a distribution on $\Omega \times (-\infty, T)$ such that
\[
\text{WF}(V) \subset \{ (x,t; \xi, \tau) \in T^*( \Omega \times (-\infty,T)) : |\tau| \geq c |\xi| \},
\]
for some $c>0$. If 
\begin{gather*}
\Box V =0  \text{~ on $(\R^n \setminus \Bbar) \times (-\infty, T)$},
\\
V|_{\pa B \times (-\infty, T)}=0, \qquad \text{$V=0$ on 
$(\Omega \setminus \Bbar) \times (-\infty, S)$ }
\end{gather*}
for some $S<T$, then 
\[
V=0  \text{~on  $(\R^n \setminus \Bbar) \times (-\infty, T)$}.
\]
\end{proposition}
Note that since WF$(V)$ does not intersect the normal bundle of $\pa B \times (-\infty, T)$, the restriction of 
$V$ to $\pa B \times (-\infty, T)$ is well defined by the manifold version of \cite[Theorem 8.2.4]{horm90} applied to the 
imbedding  $i : \pa B \times (-\infty,T) \to \Omega \times (-\infty, T)$ with $i(x,t)= (x,t)$.

The proof of Proposition \ref{prop:exterior} uses the following observation. Below, $V * \rho$ denotes convolution 
only in $t$.
\begin{lemma}\label{lemma:conv}
Suppose $\Omega$ is an open subset of $\R^n$ with $H \subset \Omega$ and $V(x,t)$ is a compactly supported distribution on 
$\Omega \times (a,b)$ such that
\begin{equation}
\text{WF}(V) \subset \Gamma:=   \{ (x,t; \xi, \tau) \in T^*( \Omega \times (a,b)) : |\tau| \geq c |\xi| \}
\label{eq:nonhorizontal}
\end{equation}
for some $c>0$. If $\rho \in C_c^\infty(-\ep,\ep)$ for a small $\ep>0$
then $V*\rho$ (convolution only in $t$) is smooth on $\Omega \times (a+\ep, b-\ep)$.
Further, $V$ has a trace on $H \times (a,b)$ and 
\begin{equation}
(V*\rho)|_{H \times (a+\ep, b-\ep)} = \rho*(V|_{H \times (a,b)})
\end{equation}
with the RHS restricted to $H \times (a+\ep, b-\ep)$.
\end{lemma}
We postpone the proof of Lemma \ref{lemma:conv} to the end of this section.


\subsection*{Proof of Proposition \ref{prop:exterior}}
Choose a $\rho \in C_c^\infty(-1,1)$ with $\int_\R \rho=1$ and, for any $\ep>0$, define the function $\rho_\ep$ by 
$\rho_\ep(t) = \ep^{-1}\rho(\ep^{-1}t)$. 
Define $V_\ep = V*\rho_\ep$ on $\Omega \times (-\infty, T-\ep)$. We claim that $V_\ep$ is smooth and
$V_\ep|_{\pa B \times (-\infty, T-\ep)}=0$. These claims are a quick consequence of of Lemma \ref{lemma:conv} because
the claims are local in nature and a local diffeomorphism in $x$ can straighten out pieces of the curved surface $\pa B$ to 
a piece of $H$. 

Hence, from the hypothesis, the smooth function $V_\ep$ on $\Omega \times (-\infty, T-\ep)$ is a solution of the IBVP
\begin{align*}
\Box V_\ep =0, & \qquad \text{on } (\R^n \setminus B) \times (-\infty, T-\ep)
\\
V_\ep =0, & \qquad \text{on } \pa B \times (-\infty, T-\ep)
\\
V_\ep=0, & \qquad  \text{on }  (\R^n \setminus B) \times (-\infty, S - \ep).
\end{align*}
Pick any $ (x_0, t_0)$ in $(\R^n \setminus \Bbar) \times (-\infty, T-\ep)$ and define the conical region
\[
K := \{ (x,t) : |x-x_0| \leq t_0-t \} \setminus (B \times (-\infty, t_0) ).
\]
Integrating over $K$, the relation
\[
2 \, \pa_t V_{\ep} \, \Box V_\ep = \pa_t [ (\pa_t V_\ep)^2 + |\nabla V_\ep|^2]- 2 \, \nabla \cdot (\pa_t V_\ep \, \nabla V_\ep),
\]
and noting that $\Box V_\ep=0$ on $K$,
one can show that $V_\ep$ is constant on the surface 
\[
\{(x,t) : t_0-t=|x-x_0| \} \cap \{ (x,t) : |x| \geq 1 \},
\]
hence $V_\ep =0$ on it. Varying $(x_0, t_0)$ we see that 
$V_\ep=0$ on $(\R^n \setminus B) \times (-\infty, T-\ep)$. 

Finally, for any $\varphi \in C_c^\infty( U \times (-\infty, T))$, since $\varphi*\rho_\ep \to \varphi$ in the topology of
$C_c^\infty(U \times (-\infty, T))$, we have $V_\ep \to V$ as distributions on 
$\Omega \times (-\infty, T')$, for any $T'<T$. So $V=0$ on $(\Omega \setminus \Bbar) \times (-\infty, T')$ for any $T'<T$.

\subsection*{Proof of Lemma \ref{lemma:conv}} The lemma follows from a standard argument.

Let
\[
\Gamma_1 = \{ (\xi, \tau) \neq 0: |\tau| \leq c |\xi|/2 \},
\qquad
\Gamma_2 = \{ (\xi, \tau) : |\tau| \geq c |\xi|/2 \}.
\]
Since $V$ is compactly supported, there is an integer $M$ and a constant $A$ such that
\begin{equation}
|\hat{V}(\xi, \tau)| \leq A (1 + |\xi| + |\tau|)^{M}, \qquad \forall (\xi, \tau) \in \R^{n+1}.
\label{eq:Uplus}
\end{equation}
From \eqref{eq:nonhorizontal},  $\Gamma_1 \cap \Gamma = \emptyset$, hence using the compactness of the support of $V$ and 
of the set of unit vectors in $\Gamma_1$, and a partition of unity argument, 
for each positive integer $N$ there is a  $B_N$ such that
\[
|\hat{V}(\xi, \tau)| \leq B_N (1 + |\xi| + |\tau|)^{-N}, \qquad \forall (\xi, \tau) \in \Gamma_1.
\]
Since $\rho$ is smooth and compactly supported, for each positive integer $N$ there is a $C_N$ such that
\begin{equation}
|\hat{\rho}(\tau)| \leq C_N (1 + |\tau|)^{-N}, \qquad \forall \tau \in \R.
\label{eq:fminus}
\end{equation}
Hence, for any positive integer $N$, for $(\xi, \tau) \in \Gamma_1$ we have
\begin{align*}
|\widehat{V*\rho}(\xi,\tau)| = |\hat{V}(\xi, \tau)| \, |\hat{\rho}(\tau)| 
\leq B_N \, C_1 (1 + |\xi| + |\tau|)^{-N}, 
\end{align*}
and for $(\xi, \tau) \in \Gamma_2$ we have 
\[
1+|\xi| + |\tau| \leq 1 + 2 |\tau|/c + |\tau| \leq C ( 1 + |\tau|)
\]
for some $C>0$, so for some constant $D_N$, we have
\begin{align*}
|\widehat{V*\rho}(\xi,\tau)| & = |\hat{V}(\xi, \tau)| \, |\hat{\rho}(\tau)| 
\\
& \leq A \, C_N (1 + |\xi| + |\tau|)^M \, (1+|\tau|)^{-N}
\\
& \leq D_N (1 + |\xi| + |\tau|)^M \, (1+|\xi| + |\tau|)^{-N}
\\
&= D_N (1 +|\xi| + |\tau|)^{M-N}.
\end{align*}
Hence $V*\rho$ is smooth.

Noting the condition \eqref{eq:nonhorizontal} in the hypothesis, \cite[Theorem 8.2.4]{horm90} applied to the 
function $f: H \times (a,b) \to \Omega \times (a,b)$ with
$f(y,t) = ( (y,0),t)$, implies that $V$ has a trace on $H \times (a,b)$. Define
\[
\Gamma'  := f^*(\Gamma) = \{ (y,t; \eta, \tau) \in T^*(H \times (a,b) ) : |\tau| \geq c |\eta| \}.
\]
Then from \cite[Theorem 8.2.3, Theorem 8.2.4]{horm90}, there is a sequence $V_j \in C_c^\infty( \Omega \times (a,b))$ with 
$V_j \to V$ in $\D'_\Gamma(\Omega \times (a,b))$ and $V_j|_{H \times (a,b)} \to V|_{H \times (a,b)}$ in 
$\D'_{\Gamma'}(H \times (a,b))$. Further, since $V$ is compactly supported, multiplying $V_j$ and $V$ by a compactly supported smooth function which is $1$ on the support of $V$, one may 
assume that the $V_j$ and $V$ are supported in the same compact subset of $\Omega \times (a,b)$.

Since $\Gamma_ 1 \cap \Gamma = \emptyset$, using the compactness of the common support of $V_j, V$ and the compactness of the set of unit vectors in $\Gamma_1$, and a partition of unity, for each positive integer $N$ 
\[
\lim_{j \to \infty} \, \sup_{\Gamma_1} \, (|\xi|+|\tau|)^N \left  | (\hat{V} - \hat{V_j})(\xi,\tau) \right | = 0.
\]
Now
\[
| \widehat{V_j * \rho} - \widehat{V*\rho}| = |\hat{V_j} - \hat{V}| \, |\hat{\rho}| \leq C |\hat{V_j} - \hat{V}|,
\]
so\footnote{
For an open subset $X$ of $\R^m$ and a closed conical subset
$C$ of $T^*(X)$, convergence in $\D'_C(X)$ is defined
(see \cite[Definition 8.2.2]{horm90}) via a localization (multiplying by a function in $C_c^\infty(X)$) and then a Fourier transform estimate. However, for distributions in 
$\D'_C(X)$ with support in fixed compact subset of $X$, convergence in $\D'_C(X)$ follows if we can establish the same estimate without the localization. This can be shown by imitating the argument in the proof of \cite[Lemma 8.1.1]{horm90}.
}
$V_j * \rho \to V*\rho$ in $\D'_{\Gamma_2}(\Omega \times (a+\ep, b-\ep))$. Hence, by \cite[Theorem 8.2.4]{horm90},
\[
V_j * \rho|_{H \times (a+\ep, b-\ep)} \to V*\rho|_{H \times (a+\ep, b-\ep)}
\qquad \text{in  } ~\D'(H \times (a+\ep, b-\ep)).
\]
Also, for any compactly supported 
$W \in \D'(H \times \R)$ and $\varphi \in C_c^\infty(H \times \R)$, one has
\[
\la W(y,t)* \rho(t), \varphi(y,t) \ra = \la W(y,t), \varphi(y,t)*\rho(-t) \ra
\]
and $\varphi(y,t)*\rho(-t) \in C_c^\infty(H \times \R)$. This observation implies that
\[
V_j|_{H \times (a,b)}*\rho \to V|_{H \times (a,b)}*\rho 
\qquad \text{in } ~ \D'(H \times (a+\ep, b-\ep)). 
\]
Since
\[
V_j|_{H \times (a,b)}*\rho = V_j*\rho|_{H \times (a+\ep, b-\ep)}
\]
with the LHS restricted to $H \times (a+\ep, b-\ep)$, the second part of the lemma has been proved.


\section{The energy estimate}

We state and prove the proposition needed in the proof of Proposition \ref{prop:Delta}.
Suppose $\alpha(x)$ is a smooth function on $\Bbar$ and $\tau,T$ are real numbers such that 
\[
\max_{x \in \Bbar} \alpha(x) <  \tau < T.
\]
Let $\Omega := \{ (x,t) : x \in \Bbar, ~ \alpha(x) < t < T \}$.

\begin{proposition}\label{prop:energy}
If $v$ is a smooth function on $\Omega$ with 
\begin{align*}
\L v =0, & \qquad \text{on } \Omega;
\\
\qquad v(\cdot, \tau)=0, ~ v_t(\cdot, \tau)=0, & \qquad \text{on } \Bbar;
\\
v=0, ~ \pa_\nu v =0, & \qquad \text{on } \Omega \cap ( \pa B \times \R );
\end{align*}
then $v=0$ on $\Omega$.
\end{proposition}

\begin{proof}
This is a standard energy estimate proof, with a minor variation because a part of the boundary is a characteristic surafce.

It will be enough to show that $v=0$ on the region $\{(x,t) : x \in \Bbar, ~ \alpha(x) + \ep \leq t \leq T-\ep \}$, for all small 
enough $\ep>0$. Hence, replacing $\alpha$ by $\alpha+\ep$ and $T$ by $T-\ep$, one sees that it is enough to prove the 
proposition with the stronger assumption that $v$ is smooth on $\bar{\Omega}$.

Let $\tau_* := \max_{x \in \Bbar} \alpha(x)$. Using a standard energy estimate argument (see below), one can show that
$v=0$ on $\Bbar \times [\tau_*, T]$. It remains to show that $v=0$ on the region
$ \{ (x,t) : x \in \Bbar, ~ \alpha(x) \leq t \leq \tau^* \}$. Since this region is contained in the union of the sets 
$\{ (x,\alpha(x)+s) : x \in \Bbar \}$, $0 \leq s \leq \tau_*$, using an argument where $\alpha$ is replaced by $\alpha+s$,
we see that it is enough to show that $v=0$ on $ \{ (x, \alpha(x)) : x \in \Bbar \}$.

Define
\[
\Omega_* := \{ (x,t) : x \in \Bbar, ~ \alpha(x) \leq t \leq \tau^* \};
\]
its boundary consists of
\[
K_* = \{ (x, \tau_*) : x \in \Bbar \}, \qquad S = \{ (x,\alpha(x)) : x \in \Bbar \}.
\]
We have the identity (using the summation convention)
\begin{align}
2  \mdetg  \, \pa_t v \, \L v  & = 
 \pa_t \left ( \mdetg \, (\pa_t v)^2 + \mdetg \, (\nabla v)^T g^{-1} \, (\nabla v) \right )
 \nn
 \\
& \qquad - 2 \pa_i \left ( \mdetg \; \pa_t v \, g^{ij} \; \pa_j v \right ).
\label{eq:eniden}
\end{align}
Therefore, using the divergence theorem, and that $\L v=0$ on $Q_*$, that $v=0, v_t=0$ on $K_*$, and that
$ [ \nabla \alpha, -1]$ is a downward point normal to $S$, we have
\begin{align*}
0 &= \int_{\Omega_*} 2m \sqrt{\det \, g} \, v_t \, \L v
 = \int_S \mdetg \, \left [ (\pa_t v)^2 + (\nabla v)^T g^{-1} \, (\nabla v) + 2 \pa_t v \, (\nabla \alpha)^T \, g^{-1} \, \nabla \alpha
\right ] \, dx
\\
& = \int_S\mdetg \, \left [  ( \pa_t v \, \nabla \alpha  + \nabla v)^T \, g^{-1} \; (\nabla \alpha \, \pa_t v + \nabla v)
\right ]  \, dx.
\end{align*}
Hence $\pa_t v \, \nabla \alpha + \nabla v =0$ on $S$. Noting that
\[
\nabla ( v(x, \alpha(x)) ) = (\nabla v + \pa_t v \, \nabla \alpha)(x, \alpha(x)), \qquad x \in \Bbar,
\]
we conclude that $v$ is constant on $S$. Since $v=0$ on the edge of $S$ (that is $S \cap (\pa B \times \R)$), we conclude 
that $v=0$ on $S$.

\end{proof}

\bibliographystyle{custom_abbrv}
\bibliography{references}

\end{document}